\documentclass[12pt,a4paper,reqno]{amsart}

\usepackage{amsmath,amssymb,amsthm}
\usepackage{mathtools}
\usepackage{geometry}
\allowdisplaybreaks
\numberwithin{equation}{section}

\newtheorem{theorem}{Theorem}[section]
\newtheorem{lemma}[theorem]{Lemma}
\newtheorem{remark}[theorem]{Remark}

\newtheorem{corollary}[theorem]{Corollary}

\newcommand{\TV}{d_{\mathrm{TV}}}

\newcommand{\E}{\mathbb E}
\newcommand{\Pp}{\mathbb P}
\renewcommand{\P}{\mathbb P}
\newcommand{\R}{\mathbb R}

\newcommand{\op}{\mathrm{op}}

\newcommand{\tr}{\mathrm{trace}}

\title[Propagation of Chaos and its breakdown
for the Hopfield Model]{High-Temperature Increasing Propagation of Chaos and its breakdown\\
	for the Hopfield Model}
\author[Matthias L\"owe]{Matthias L\"owe}
\address[Matthias L\"owe]{Fachbereich Mathematik und Informatik,
	University of M\"unster,
	Einsteinstra\ss e 62,
	48149 M\"unster,
	Germany}
\date{}

\begin{document}
	\begin{abstract}
		We analyze increasing propagation of chaos in the high temperature regime of a disordered mean-field model, the Hopfield model. We show that for $\beta<1$ (the true high temperature region) we have increasing propagation of chaos as long as the size of the marginals $k=k(N)$ and the number of patterns 
		$M=M(N)$ satisfies $Mk/N \to 0$. For $M=o(\sqrt N)$ we show that propagation of chaos breaks down for $k/N \to c>0$. At the ciritcal temperature we show that, for $M$ finite, there
		is increasing propagation of chaos, for $k=o(\sqrt N)$, while we have breakdown of 
		propagation of chaos for $k=c \sqrt N$, for a $c>0$. All these reulst hold in probability in the disorder.   
	\end{abstract}
	\subjclass[2020]{Primary: 82B05, 82B44 Secondary: 82B20, 60F05}
	\keywords{Hopfield model, disordered systems, propagation of chaos, total variation distance, mixture distribution}
	\thanks{Research was funded by the Deutsche Forschungsgemeinschaft (DFG, German Research Foundation) under Germany's Excellence Strategy EXC 2044 - 390685587, Mathematics M\"unster: \emph{Dynamics-Geometry-Structure}.}
	
	\maketitle
	
	\section{Introduction}
	Propagation of chaos is a central concept in the study of interacting
	particle systems and mean-field models.
	Roughly speaking, it describes the phenomenon that,
	as the system size $N$ tends to infinity,
	finite collections of particles behave asymptotically independently,
	with a common limiting law.
	Propagation of chaos originated in Kac's Markovian models for gas dynamics
	\cite{Kac_foundations,Kac_probability}, an attempt to justify
	Boltzmann's ''Sto\ss zahlenansatz''.
	Propagation of chaos has since become an important 
	object of study
	in probability, statistical mechanics, and mathematical physics.
	The original approach by Kac was that,
	if at time 0 the finite marginal distributions of a system are product measures in the
	thermodynamic limit, then this should carry over to the time-evolved system. In equilibrium settings for mean-field
	Gibbs measures (where the energy function is a 
	function of the empirical measure) with a unique minimizer of the Helmholtz free energy, however, this was shown to
	follow from the fact the extremal Gibbs measures locally look like product measures, i.e.\
	that any finite subset of spins forms a family of independent random variables in the thermodynamic limit \cite[Theorem 3]{BAZ_chaos}.
	Such results provide a rigorous justification of mean-field
	approximations and explain why macroscopic behavior
	can often be described by effective one-particle models. In the present note, we remain entirely within a static framework, and study propagation of chaos in the sense of asymptotic factorization of finite-dimensional marginals for families of (random) 
	mean-field Gibbs measures.
	
	Moreover, for mean-field spin systems, propagation of chaos
	is closely tied to the high-temperature regime.
	In classical models such as the Curie--Weiss Ising model,
	chaos holds when the inverse temperature is below the critical value,
	while it breaks down in the low-temperature phase. There one has to replace
	the product measure by a mixture of the (several) extremal limiting Gibbs measures.
	An additional difficulty may arise when the Gibbs measures in question are random. To the best of our knowledge comparatively little is known in such situations (with the exception of of \cite{BG98book} and \cite{KL24}). The present paper contributes to the study of propagation of chaos for random mean-field Gibbs measures in the high-temperature and near-critical regimes.

	Importantly, in \cite{BAZ_chaos}, the authors also introduced the concept of \textit{increasing} propagation of chaos and showed that in the true high temperature regime of many mean-field models one can let the size of the marginals $k$ grow with the system size $N$ as long as $k=o(N)$. 
	Understanding not only whether propagation of chaos holds,
	but also the precise scales at which it breaks down,
	has become an important theme in recent work on mean-field Gibbs measures \cite{BAZ_chaos,Lacker22, JKLM23,KL24,RS25,JKL25}.
	
	While increasing propagation of chaos has been established in several genuinely high-temperature regimes of ordered mean-field models, much less is known in the presence of random mean-field Gibbs measures, where already the analysis of critical fluctuations can become delicate (\cite{gentzloewe, gentzloewe2, talagrand_critical_hopf}).   
	In the present setting of random mean-field Gibbs measures, we show that this high-temperature scaling persists for $\beta <1$ (the critical value),
	whereas at criticality 
	$\beta=1$, the admissible growth drops to $k=o(\sqrt N)$. Moreover, 
	these bounds on 
	$k$ are optimal for the class of random mean-field Gibbs measures considered here.
	
	In this article we study propagation of chaos
	in the Hopfield model, a paradigmatic example
	of a disordered mean-field spin system.
	The Hopfield model has (at least) two distinct origins. It was first introduced by
	Pastur and Figotin \cite{PS84} as a solvable model of a disordered system. At about the same time it
	was independently invented by Hopfield as
	a model of associative memory \cite{Hopfield1982}.
	Both aspects have been intensively studied. For the disordered systems facet see e.g. \cite{BGP94, BG97} or \cite{talagrand}, while 
	the neural network aspect has found a renewed interest through two recent papers \cite{KrotovHopfield2016,DHLUV17}.
	
	For the purposes of the present paper, the 
	probabilistic viewpoint of seeing the Hopfield model
	as a mean-field Ising model with random, structured interactions is more appropriate.
	The disorder, i.e.\ the random interactions, is generated by a collection of random patterns,
	which induces a random quadratic Hamiltonian
	and leads to a rich interplay between thermal fluctuations
	and quenched randomness.
	While the thermodynamic properties of the Hopfield model
	are well understood, much less is known about
	the fine structure of its finite-dimensional marginals
	and their asymptotic independence properties. There is only one result concerning the propagation of chaos in the Hopfield model, see \cite[Theorem 8.15]{BG98}. However, this result addresses fixed-dimensional marginals in the low-temperature regime, and relies on mechanisms different from those considered here.
	
	Our main goal is to analyze propagation of chaos
	for the Hopfield Gibbs measure in the high-temperature
	and critical regimes, with particular emphasis
	on the size of the marginals.
	We identify precise conditions under which propagation of chaos holds,
	as well as sharp thresholds for its breakdown.
	In the high-temperature regime, we prove propagation of chaos
	for growing marginals, provided their size grows sublinearly in the system size $N$.
	At criticality, we show that chaos breaks down in a critical window,
	whose scale matches that of the dominant collective fluctuations.
	reveal a clear transition between asymptotic independence and regimes of partial or complete breakdown of chaos,
	and highlight the role played by disorder-induced fluctuations
	in determining these regimes.
	
	\subsection{The Model}
	Let us next describe the central model for the purpose of this note.
	Let $\xi_i=(\xi_i^1,\dots,\xi_i^M)\in\{-1,+1\}^M$, $i=1,\dots,N$, be i.i.d.\ random vectors with independent coordinates,
	\[
	\E[\xi_i^\nu]=0, \qquad \E[(\xi_i^\nu)^2]=1 .
	\]
	(In the interpretation of an associative memory, the vectors $(\xi^\mu)_{\mu=1}^M=((\xi_i^\mu)_{i=1}^N)_{\mu=1}^M$ are called images or patterns).
	In what follows we will always assume  that $M=M(N)$ may depend on $N$, but in such a way, that $M=o(N)$, which is the natural regime in which the Hopfield model exhibits mean-field behavior. 
	For fixed, i.e.\ quenched, patterns $(\xi^\mu)_\mu$	define and for $\sigma\in\{-1,+1\}^N$, define the overlap vector
	\[
	m_N(\sigma) = \big(m_N^1(\sigma),\dots,m_N^M(\sigma)\big),
	\qquad
	m_N^\nu(\sigma)=\frac1N\sum_{i=1}^N \sigma_i\xi_i^\nu .
	\]
	The Hopfield Gibbs measure at inverse temperature $\beta>0$ is
	\[
	\mu_N(\sigma)=\frac{1}{Z_N}
	\exp\Big(\frac{\beta N}{2}\|m_N(\sigma)\|^2\Big).
	\]
	The Gibbs measure is fully determined by $m_N(\sigma)$ (plus the inverse temperature), it thus is natural to consider 
	the behaviour of the overlap under the Gibbs measure.
	This has been done in numerous papers: 
	In \cite{BGP94} it was shown that -- similar to the Curie-Weiss model -- in the Hopfield model the critical (inverse) temperature is $\beta=1$. While for $\beta \le 1$ the overlap vector  gets concentrated in the $M$-dimensional 0-vector, for larger $\beta$ the limit points of $m_N$ are associated with the 2$M$ vectors $\pm z(\beta) e_\mu$ where 
	$e_\mu$ is the $\mu$'th unit vector and $z(\beta)$ is the largest solution of the equation
	$$
	z=\tanh(\beta z).
	$$
	Central Limit theorems for $\sqrt N m_N(\cdot)$ were proven by Gentz \cite{gentz_annals} or Bovier and Gayrard \cite{BG_CLT}. All these results hold true for almost all realizations of the patterns and only need $M=o(N)$
	(for \cite{gentz_annals} this is only true in the high temperature regime, which, however, is the most relevant for us in the present note.) Importantly, as in the Curie-Weiss model, at $\beta=1$ the fluctuations are non-Gaussian. As was shown in \cite{gentzloewe, gentzloewe2, talagrand_critical_hopf} at $\beta=1$
	the rescaled overlap vector $N^{1/4} m_N$ converges in distribution to a random limit (while in the other limit theorems the limit was deterministic). Also large and moderate deviations for the overlap vector are available (see \cite{BG_LDP, EL_hopf}).

	\subsection{Statement of the results}
	In this subsection we state our main results on increasing propagation of chaos
	for the Hopfield model in the high-temperature and critical regimes. In particular, we identify regimes in which propagation of chaos breaks down
	at explicit scales of the marginals. For $\beta<1$, we allow the number of patterns $M=M(N)$ to diverge with $N$,
	and establish both increasing propagation of chaos and sharp breakdown results.
	At criticality $\beta=1$, we restrict attention to fixed $M$, and identify the
	critical window in which propagation of chaos fails.

	In order to formulate our results, let us agree on the following notation:  
	For $k\le N$, let $\mu_N^{(k)}$ denote the marginal of $\mu_N$ on
	$(\sigma_1,\dots,\sigma_k)$.
	Let $\pi$ be the Rademacher law on $\{-1,+1\}$ with $\pi(\pm1)=\frac12$.
	
	Then we will prove: 
	\begin{theorem}
		\label{thm:highT}
		Fix $\beta<1$. Let $M=M(N)$ satisfy $M/N\to0$, and let $k=k(N)\to\infty$
		with
		\[
		\frac{k(N)M(N)}{N}\longrightarrow 0 
		\]
		(which in particular allows $k=o(N)$ when $M$ is fixed).
		
		Then
		\[
		\TV\big(\mu_N^{(k)},\pi^{\otimes k}\big)
		\longrightarrow 0
		\qquad\text{in }\Pp_\xi\text{-probability}.
		\]
	\end{theorem}
	
	\begin{theorem}
		\label{thm:stop_macro_Msmall}
		Fix $\beta\in(0,1)$ and assume $M=M(N)\to\infty$ with $M=o(\sqrt N)$.
		Let $k=k(N)$ satisfy $k/N\to\rho\in(0,1)$.
		Then, in $\P_\xi$-probability,
		\[
		\TV\big(\mu_N^{(k)},\pi^{\otimes k}\big)\longrightarrow 1.
		\]
	\end{theorem}
	
	\begin{remark}
		The restriction $M=o(\sqrt N)$ in Theorem~\ref{thm:stop_macro_Msmall}
		is stronger than in Theorem~\ref{thm:highT}.
		Whether this condition can be relaxed remains an open problem.
	\end{remark}	
	
	\begin{theorem}
		\label{thm:crit_poc}
		Fix $\beta=1$ and let $M\in\mathbb N$ be fixed. Assume $k=k(N)\to\infty$ satisfies
		\[
		\frac{k(N)}{N^{1/2}}\longrightarrow 0.
		\]
		Then
		\[
		d_{\mathrm{TV}}\big(\mu_N^{(k)},\pi^{\otimes k}\big)\longrightarrow 0
		\qquad\text{in }\mathbb P_\xi\text{-probability}.
		\]
	\end{theorem}
	
	And finally, we show
	\begin{theorem}[Critical-window breakdown at $\beta=1$ (fixed $M$)]
		\label{thm:crit_break_sqrtN}
		Fix $\beta=1$ and let $M\in\mathbb N$ be fixed.
		Let $k=k(N)$ satisfy
		\[
		\frac{k}{\sqrt N}\longrightarrow c\in(0,\infty).
		\]
		Then there exists a deterministic constant $b=b(c,M)>0$ such that
		\[
		\liminf_{N\to\infty} \ d_{\mathrm{TV}}\big(\mu_N^{(k)},\pi^{\otimes k}\big)\ \ge\ b
		\qquad\text{in }\mathbb P_\xi\text{-probability}.
		\]
		In particular, propagation of chaos fails in the critical window $k\asymp \sqrt N$,
		showing that the scaling in Theorem~\ref{thm:crit_poc} is optimal.
		
	\end{theorem}
	\subsection{Outline of the proofs}
	
	A central idea in our proofs is that the Hopfield Gibbs measure admits,
	via a Hubbard--Stratonovich transformation, an explicit representation
	in terms of a mixture of product measures.
	
	\medskip
	\noindent\textit{Step 1: Mixture-of-products representation.}
	For fixed (quenched) patterns $(\xi^\mu)_\mu$, the quadratic Hamiltonian
	$\frac{\beta N}{2}\|m_N(\sigma)\|^2$ can be linearized by introducing an
	auxiliary Gaussian field. This standard technqie for quadratic Hamiltonians yields a representation of the form
	\[
	\mu_N(d\sigma)\;=\;\int \nu_{N,\beta}(d y)\,
	\bigotimes_{i=1}^N \mu_{y,i}(d\sigma_i),
	\]
	where $y\in\R^M$ is the Hubbard--Stratonovich field, $\nu_{N,\beta}$ is an
	explicit probability measure on $\R^M$, and $\mu_{y,i}$ is a Bernoulli law
	with bias depending on $y\cdot \xi_i$.
	Consequently, the $k$-spin marginal $\mu_N^{(k)}$ is a mixture of product
	measures on $\{-1,+1\}^k$.
	
	\medskip
	\noindent\textit{Step 2: Propagation of chaos reduces to stability of the mixture.}
	Conditionally on $y$, the spins are independent.
	Thus propagation of chaos for $\mu_N$ is controlled by how strongly the
	random field $y$ fluctuates under $\nu_{N,\beta}$ and by how sensitively
	the single-spin biases $\mu_{y,i}$ depend on $y$.
	Quantitatively, we compare the mixture to the unbiased product law
	$\pi^{\otimes k}$ in total variation distance.
	
	\medskip
	\noindent\textit{Step 3: High temperature $\beta<1$.}
	In the regime $\beta<1$, the Hubbard--Stratonovich field typically remains
	of order one and concentrates near the origin.
	A Taylor expansion of the single-spin biases and a control of
	$\nu_{N,\beta}$ show that correlations between $k$ spins are of order
	$kM/N$, yielding Theorem~\ref{thm:highT} under the condition $kM/N\to0$.
	For macroscopic $k$ (with $k/N\to\rho$), the same representation allows us
	to exhibit order-one correlations induced by the random field, which leads
	to strong breakdown in total variation (Theorem~\ref{thm:stop_macro_Msmall})
	under the stated assumptions on $M$.
	
	\medskip
	\noindent\textit{Step 4: Critical temperature $\beta=1$ and the critical window.}
	At criticality, the mixing measure $\nu_{N,1}$ develops non-Gaussian
	fluctuations on the scale $N^{-1/4}$, reflecting the well-known
	critical behavior of the overlap.
	This amplification of the mixing fluctuations reduces the admissible growth
	of the marginals to $k=o(N^{1/4})$ (Theorem~\ref{thm:crit_poc}).
	Moreover, when $k\asymp \sqrt N$, the mixture retains a nontrivial
	amount of randomness that produces correlations bounded away from zero,
	yielding a breakdown of propagation of chaos in the critical window
	(Theorem~\ref{thm:crit_break_sqrtN}) and hence the optimality of the scaling.

	\section{Preliminaries}
	
	\subsection{Mixture representation}
	A key tool in our proofs will be a representation
	of the Gibbs measure as a mixture of products of 
	Bernoulli measures. This is a direct consequence of the Hubbard–Stratonovich transformation.

	\begin{lemma}[Hubbard--Stratonovich mixture]
		\label{lem:HS}
		Fix $\beta>0$. For every realization of $\xi$, there exists a probability measure $Q_N$
		on $\R^M$ such that
		\[
		\mu_N(d\sigma)
		=
		\int_{\R^M} Q_N(du)
		\bigotimes_{i=1}^N \mu_i^{(u)}(d\sigma_i),
		\]
		where 	is the Bernoulli measure on $\{-1,+1\}$ given by
		\[
		\mu_i^{(u)}(\sigma_i)
		=
		\frac{\exp(\sigma_i\,u\cdot\xi_i)}
		{2\cosh(u\cdot\xi_i)},
		\qquad \sigma_i\in\{-1,+1\}
		\]
		and
		\[
		Q_N(du)
		=
		\frac{1}{\mathcal Z_N}
		\exp\Big(-\frac{N}{2\beta}\|u\|^2\Big)
		\prod_{i=1}^N 2\cosh(u\cdot\xi_i)\,du,
		\]
		and $\mathcal Z_N$ is the normalizing constant to turn $Q_N$ into a probability measure: 
		$$
		\mathcal Z_N = \int_{\R^M} \exp\Big(-\frac{N}{2\beta}\|u\|^2\Big)
		\prod_{i=1}^N 2\cosh(u\cdot\xi_i)\, du.
		$$
	\end{lemma}

	\begin{proof}
		We start from the definition of the Hopfield Gibbs measure
		\[
		\mu_N(\sigma)
		=
		\frac{1}{Z_N}
		\exp\Big(\frac{\beta N}{2}\|m_N(\sigma)\|^2\Big),
		\qquad
		\sigma\in\{-1,+1\}^N,
		\]
		where
		$
		m_N(\sigma)=\frac1N\sum_{i=1}^N \sigma_i \xi_i \in \mathbb R^M .
		$
		
		Now by the Hubbard-Stratonovich transformation (i.e.\ completing the square in the exponents)
		\[
		\exp\Big(\frac{\beta N}{2}\|m_N(\sigma)\|^2\Big)
		=
		\Big(\frac{N}{2\pi\beta}\Big)^{M/2}
		\int_{\mathbb R^M}
		\exp\Big(
		-\frac{N}{2\beta}\|u\|^2 + N u\cdot m_N(\sigma)
		\Big)\,du .
		\]
		Moreover, by definition of $m_N(\sigma)$,
		$N u\cdot m_N(\sigma)=
		\sum_{i=1}^N \sigma_i \, (u\cdot \xi_i),$
		where $\xi_i$ is the M-dimensional vector $(\xi_i^\nu)_{\nu=1}^M$. 
		Hence
		\[
		\exp\Big(\frac{\beta N}{2}\|m_N(\sigma)\|^2\Big)
		=
		\Big(\frac{N}{2\pi\beta}\Big)^{M/2}
		\int_{\mathbb R^M}
		\exp\Big(
		-\frac{N}{2\beta}\|u\|^2
		\Big)
		\prod_{i=1}^N
		\exp\big(\sigma_i (u\cdot \xi_i)\big)
		\,du .
		\]
		Substituting the previous representation into the partition function and using Fubini's theorem, we obtain
		\begin{align*}
			Z_N
			=&\,
			\Big(\frac{N}{2\pi\beta}\Big)^{M/2}
			\int_{\mathbb R^M}
			\exp\Big(-\frac{N}{2\beta}\|u\|^2\Big)
			\sum_{\sigma\in\{-1,+1\}^N}
			\prod_{i=1}^N \exp\big(\sigma_i (u\cdot \xi_i)\big)
			\,du \\
			=&\,
			\Big(\frac{N}{2\pi\beta}\Big)^{M/2}
			\int_{\mathbb R^M}
			\exp\Big(-\frac{N}{2\beta}\|u\|^2\Big)
			\prod_{i=1}^N 2\cosh(u\cdot \xi_i)
			\,du 
		\end{align*}
		Returning to $\mu_N(\sigma)$ and inserting the same representation in the numerator, we obtain
		\[
		\mu_N(\sigma)
		=
		\int_{\mathbb R^M}
		\frac{
			\exp\big(-\frac{N}{2\beta}\|u\|^2\big)
			\prod_{i=1}^N \exp(\sigma_i (u\cdot \xi_i))
		}{
			\int_{\mathbb R^M}
			\exp\big(-\frac{N}{2\beta}\|v\|^2\big)
			\prod_{i=1}^N 2\cosh(v\cdot \xi_i)\,dv
		}
		\,du .
		\]
		Writing
		\[
		\exp(\sigma_i (u\cdot \xi_i))
		=
		\frac{\exp(\sigma_i (u\cdot \xi_i))}{2\cosh(u\cdot \xi_i)}
		\cdot 2\cosh(u\cdot \xi_i),
		\]
		we may rewrite $\mu_N$ as
		\[
		\mu_N(d\sigma)
		=
		\int_{\mathbb R^M} Q_N(du)
		\bigotimes_{i=1}^N \mu_i^{(u)}(d\sigma_i),
		\]
		where
		\[
		\mu_i^{(u)}(\sigma_i)
		=
		\frac{\exp(\sigma_i (u\cdot \xi_i))}{2\cosh(u\cdot \xi_i)},
		\]
		and where $Q_N$ is the probability measure on $\mathbb R^M$ with density
		\[
		Q_N(du)
		=
		\frac{
			\exp\big(-\frac{N}{2\beta}\|u\|^2\big)
			\prod_{i=1}^N 2\cosh(u\cdot \xi_i)
		}{
			\int_{\mathbb R^M}
			\exp\big(-\frac{N}{2\beta}\|v\|^2\big)
			\prod_{i=1}^N 2\cosh(v\cdot \xi_i)\,dv
		}
		\,du .
		\]
		This completes the proof.
	\end{proof}
	\begin{corollary}
		For $k\le N$, the $k$-marginal admits the representation
		\[
		\mu_N^{(k)}(d\sigma_1\cdots d\sigma_k)
		=
		\int_{\R^M} Q_N(du)\, \bigotimes_{i=1}^k \mu_i^{(u)}(d\sigma_i).
		\]
	\end{corollary}

	\subsection{Total variation bound for the $k$-marginal}
	
	As we will always be interested in the total variation distance between
	the distribution of the marginal spins and a (Rademacher) product measure,
	the following consideration is important.

	\medskip
	Let $\mathbb P^{(u)}_{[k]}$ denote the law of
	$(\sigma_1,\dots,\sigma_k)$ under the measure
	$\bigotimes_{i=1}^N \mu_i^{(u)}$ (restricted to the first $k$ coordinates) from the previous subsection.
	
	\begin{lemma}
		\label{lem:TV}
		For every $u\in\R^M$,
		$$
		\TV\big(\mathbb P^{(u)}_{[k]},\pi^{\otimes k}\big)
		\le
		\sqrt{\sum_{i=1}^k (u\cdot\xi_i)^2}.
		$$
		Consequently,
		\[
		\TV\big(\mu_N^{(k)},\pi^{\otimes k}\big)
		\le
		\E_{Q_N}\left[\sqrt{\sum_{i=1}^k (u\cdot\xi_i)^2}\right].
		\]
	\end{lemma}
	
	\begin{proof}
		For each $i$, $\mu_i^{(u)}$ is a Bernoulli law on $\{-1,+1\}$
		with mean $\tanh(u\cdot\xi_i)$. Indeed, we have
		$\mu_i^{(u)}(\sigma_i)
		=
		\frac{\exp(\sigma_i h_i)}{2\cosh(h_i)}$, where 
		$h_i := u \cdot \xi_i$. Hence
		\begin{align*}
			\mathbb E_{\mu_i^{(u)}}[\sigma_i]=		
			\frac{e^{h_i}-e^{-h_i}}{e^{h_i}+e^{-h_i}}
			=
			\tanh(h_i)
			=
			\tanh(u\cdot\xi_i).
		\end{align*}
		
		We recall that for two probability measures $\nu,\varrho$ on a finite set
		$\mathcal X$ with $\nu\ll \varrho$, the Kullback--Leibler divergence
		(relative entropy) is defined by
		\[
		\mathrm{KL}(\nu\|\varrho)
		:=
		\sum_{x\in\mathcal X}\nu(x)\log\frac{\nu(x)}{\varrho(x)} \in [0,\infty).
		\]
		In particular, for $\mathcal X=\{-1,+1\}$ and $\pi(\pm1)=\frac12$, we have
		$$
		\mathrm{KL}(\mu_i^{(u)}\|\pi)
		=
		\sum_{\sigma_i=\pm1} \mu_i^{(u)}(\sigma_i)\,
		\log\Big( 2\,\mu_i^{(u)}(\sigma_i)\Big).
		$$
		
		Let again $h_i:=u\cdot\xi_i$ and recall that
		$
		\mu_i^{(u)}(\sigma_i)
		=
		\frac{e^{\sigma_i h_i}}{2\cosh(h_i)}.
		$
		Then
		$
		2\,\mu_i^{(u)}(\sigma_i)
		=
		\frac{e^{\sigma_i h_i}}{\cosh(h_i)}$,
		and hence
		$\log\Big(2\,\mu_i^{(u)}(\sigma_i)\Big)
		=
		\sigma_i h_i - \log\cosh(h_i).
		$
		Plugging this into the definition of $\mathrm{KL}$ gives
		\[
		\begin{aligned}
			\mathrm{KL}(\mu_i^{(u)}\|\pi)
			&=
			\sum_{\sigma_i=\pm1}
			\mu_i^{(u)}(\sigma_i)\,
			\big(\sigma_i h_i-\log\cosh(h_i)\big)
			\\
			&=
			h_i \sum_{\sigma_i=\pm1}\sigma_i\,\mu_i^{(u)}(\sigma_i)
			\;-\;
			\log\cosh(h_i)\sum_{\sigma_i=\pm1}\mu_i^{(u)}(\sigma_i)
			\\
			&=
			h_i\,\E_{\mu_i^{(u)}}[\sigma_i] - \log\cosh(h_i).
		\end{aligned}
		\]
		Since (as computed above) $\E_{\mu_i^{(u)}}[\sigma_i]=\tanh(h_i)$, we obtain
		\[
		\mathrm{KL}(\mu_i^{(u)}\|\pi)
		=
		h_i\tanh(h_i) - \log\cosh(h_i).
		\]
		
		\medskip
		
		Now for all $x\in\mathbb R$,
		\begin{equation}\label{eq:tanhineq}
			h\tanh(h)-\log\cosh(h) \le 2\tanh^2(h).
		\end{equation}
		Indeed, define $f(h)$ as:
		
		$f(h) = 2\tanh^2(h) - \left( h\tanh(h) - \log\cosh(h) \right).$
		By symmetry if suffices to show $f(h) \ge 0$ for all $h \ge 0$. Note that $f(0) = 0$
		and
		$
		f'(h) = (4\tanh(h) - h)\operatorname{sech}^2(h).$
		Since $\operatorname{sech}^2(h) > 0$ for all $h$, the sign of $f'(h)$ is determined by the term $g(h) = 4\tanh(h) - h$.
		Now, at $h=0$ also $g(0) = 0$, implying $f'(0) = 0$. For $h$ small enough,
		$g(h) \approx 4h - h = 3h > 0$. Thus, $f(h)$ is initially increasing. On the other hand, as $h \to \infty$, $g(h)$ eventually becomes negative, and $f(h)$ will eventually decrease.
		To ensure $f(h)$ remains non-negative, note that for $h \to \infty$ we have $\tanh(h) \to 1$ and $\log\cosh(h) \approx h - \log 2$ and thus
		\begin{align*}
			\lim_{h \to \infty} f(h) &= \lim_{h \to \infty} \left[ 2\tanh^2(h) - (h\tanh(h) - \log\cosh(h)) \right] \\
			&= 2 - \lim_{h \to \infty} (h - (h - \log 2)) >0
		\end{align*}
		
		This shows \eqref{eq:tanhineq}.
		
		In particular,
		\[
		\mathrm{KL}(\mu_i^{(u)}\|\pi)\le 2\tanh^2(h_i)\le 2h_i^2 = 2(u\cdot\xi_i)^2.
		\]
		
		Recall that  $\mathbb P^{(u)}_{[k]}:=\bigotimes_{i=1}^k \mu_i^{(u)}$ and that
		relative entropy tensorizes over products, i.e.
		$
		\mathrm{KL}\big(\mathbb P^{(u)}_{[k]}\,\big\|\,\pi^{\otimes k}\big)
		=
		\sum_{i=1}^k \mathrm{KL}(\mu_i^{(u)}\|\pi).
		$
		Moreover, Pinsker's inequality states that for any probability measures
		$\nu,\varrho$,
		\[
		\TV(\nu,\varrho)\le \sqrt{\frac12\,\mathrm{KL}(\nu\|\varrho)}.
		\]
		Applying Pinsker with $\nu=\mathbb P^{(u)}_{[k]}$ and $\varrho=\pi^{\otimes k}$
		and using the previous bounds yields
		\begin{equation}\label{eq:boundTV}
			\TV\big(\mathbb P^{(u)}_{[k]},\pi^{\otimes k}\big)
			\le
			\sqrt{\frac12 \sum_{i=1}^k \mathrm{KL}(\mu_i^{(u)}\|\pi)}
			\le
			\sqrt{\sum_{i=1}^k (u\cdot\xi_i)^2}.
		\end{equation}
		
		By Lemma~\ref{lem:HS}, the $k$-spin marginal of
		$\mu_N$ satisfies
		$$
		\mu_N^{(k)}(\cdot)
		=
		\int_{\mathbb R^M} Q_N(du)\,\mathbb P^{(u)}_{[k]}(\cdot),
		$$
		i.e.\ $\mu_N^{(k)}$ is a convex combination (mixture) of the measures
		$\mathbb P^{(u)}_{[k]}$.
		Recall the characterization of total variation as
		\[
		\TV(\nu,\varrho)=\sup_{A\subseteq\{-1,+1\}^k}\big|\nu(A)-\varrho(A)\big|.
		\]
		Moreover, for any measurable $A\subseteq\{-1,+1\}^k$,
		\[
		\mu_N^{(k)}(A)-\pi^{\otimes k}(A)
		=
		\int Q_N(du)\,\big(\mathbb P^{(u)}_{[k]}(A)-\pi^{\otimes k}(A)\big),
		\]
		and therefore
		\[
		\big|\mu_N^{(k)}(A)-\pi^{\otimes k}(A)\big|
		\le
		\int Q_N(du)\,\big|\mathbb P^{(u)}_{[k]}(A)-\pi^{\otimes k}(A)\big|
		\le
		\int Q_N(du)\,\TV\big(\mathbb P^{(u)}_{[k]},\pi^{\otimes k}\big).
		\]
		Taking the sup over all $A$ yields the convexity bound
		\[
		\TV\big(\mu_N^{(k)},\pi^{\otimes k}\big)
		\le
		\int Q_N(du)\,\TV\big(\mathbb P^{(u)}_{[k]},\pi^{\otimes k}\big).
		\]
		
		It remains to bound $\TV\big(\mathbb P^{(u)}_{[k]},\pi^{\otimes k}\big)$ which is done by \eqref{eq:boundTV}.
		\[
		\TV\big(\mu_N^{(k)},\pi^{\otimes k}\big)
		\le
		\int Q_N(du)\,\sqrt{\sum_{i=1}^k (u\cdot\xi_i)^2}
		=
		\E_{Q_N}\Big[\sqrt{\sum_{i=1}^k (u\cdot\xi_i)^2}\Big],
		\]
		as claimed.
	\end{proof}
	
	Lemma~\ref{lem:TV} reduces the problem of propagation of chaos
	to estimating moments of the random variable
	$\sum_{i=1}^k (u\cdot\xi_i)^2$ under the mixing measure $Q_N$,
	which will be analyzed separately in the high-temperature
	and critical regimes.

	\subsection{Control of the mixing measure at high temperature}
	We next establish a Gaussian domination bound for the mixing measure $Q_N(du)$
	on a disorder event (i.e.\ an event formulated in the patterns $(\xi)$) of high probability.

	Let
	\[
	\widehat\Sigma_N = \frac1N\sum_{i=1}^N \xi_i\xi_i^\top .
	\]
	
	\begin{lemma}
		\label{lem:QN_domination}
		Fix $\beta<1$ and choose $\varepsilon>0$ such that $\beta(1+\varepsilon)<1$.
		On the event
		\[
		\mathcal E_N(\varepsilon)
		=
		\{\|\widehat\Sigma_N\|_{\mathrm{op}}\le 1+\varepsilon\},
		\]
		there exist constants $\alpha>0$ and $C_0<\infty$
		(depending only on $\beta,\varepsilon$) such that
		\[
		Q_N(A)
		\le
		C_0\,\gamma_\alpha(A)
		\qquad
		\text{for all measurable } A\subseteq\mathbb R^M,
		\]
		where $\gamma_\alpha=\mathcal N(0,(2\alpha N)^{-1}I_M)$.
	\end{lemma}
	
	\begin{proof}
		Recall that
		\[
		Q_N(du)
		=
		\frac{1}{\mathcal Z_N}
		\exp\Big(-\frac{N}{2\beta}\|u\|^2\Big)
		\prod_{i=1}^N 2\cosh(u\cdot\xi_i)\,du .
		\]
		Using the bound $\log\cosh x\le x^2/2$, we obtain
		\[
		\prod_{i=1}^N 2\cosh(u\cdot\xi_i)
		\le
		2^N
		\exp\Big(\frac12\sum_{i=1}^N (u\cdot\xi_i)^2\Big).
		\]
		On $\mathcal E_N(\varepsilon)$,
		\[
		\sum_{i=1}^N (u\cdot\xi_i)^2
		=
		N\,u^\top\widehat\Sigma_N u
		\le
		N(1+\varepsilon)\|u\|^2.
		\]
		Therefore,
		\[
		Q_N(du)
		\le
		\frac{2^N}{\mathcal Z_N}
		\exp\Big(
		-\frac{N}{2}\Big(\frac1\beta-(1+\varepsilon)\Big)\|u\|^2
		\Big)\,du .
		\]
		Since $\beta(1+\varepsilon)<1$, the coefficient
		\[
		\alpha := \tfrac12\Big(\frac1\beta-(1+\varepsilon)\Big)>0,
		\]
		and hence the right-hand side is dominated by the density of a centered
		Gaussian measure $\gamma_\alpha=\mathcal N(0,(2\alpha N)^{-1}I_M)$,
		up to a multiplicative constant $C_0<\infty$ coming from normalization.
		
		As a matter of fact, since 
		$\tilde q_N(u)\le \tilde g(u)$  pointwise, where
		\[Q_N(du)=\tilde q_N(u)Z_Q^{-1}du \qquad  \text{as well as  }
		\gamma_\alpha(du)=\tilde g(u)Z_G^{-1}du,\]
		we obtain for all measurable $A\subseteq\R^M$,
		\[
		Q_N(A)
		=
		\int_A \frac{\tilde q_N(u)}{Z_Q}\,du
		\le
		\frac{Z_G}{Z_Q}\int_A \frac{\tilde g(u)}{Z_G}\,du
		=
		C_0\,\gamma_\alpha(A),
		\]
		where $C_0:=Z_G/Z_Q<\infty$.
		Thus $Q_N$ is dominated by $\gamma_\alpha$ up to a multiplicative constant.

		This proves the claim.
	\end{proof}

	\begin{lemma}
		\label{lem:u}
		Fix $\beta<1$ and let $\varepsilon>0$ satisfy $\beta(1+\varepsilon)<1$.
		On the event
		\[
		\mathcal E_N(\varepsilon)
		=
		\{\|\widehat\Sigma_N\|_{\mathrm{op}}\le 1+\varepsilon\},
		\]
		there exists $C=C(\beta,\varepsilon)$ such that
		\[
		\E_{Q_N}\|u\|
		\le
		C\sqrt{\frac{M}{N}} .
		\]
		Moreover, if $M/N\to0$, then
		$\Pp_\xi(\mathcal E_N(\varepsilon))\to1$.
	\end{lemma}
	
	\begin{proof}
		The first claim $\E_{Q_N}\|u\|
		\le
		C\sqrt{\frac{M}{N}}$ is a direct consequence of the Gaussian domination Lemma \ref{lem:QN_domination}. Indeed, under $\gamma_\alpha$, $\|u\|$ has expectation of order $\sqrt{M/N}$,
		and domination transfers this bound to $Q_N$.

		The final claim follows from standard concentration bounds for
		sample covariance matrices when $M/N\to0$ (see Lemma \ref{lem:covconcentration} below).
	\end{proof}
	
	\begin{lemma}
		\label{lem:covconcentration}
		Let $\xi_1,\dots,\xi_N\in\mathbb R^M$ be i.i.d.\ centered subgaussian random vectors
		with covariance $\Sigma:=\E[\xi_1\xi_1^\top]$.
		Let
		\[
		\widehat\Sigma_N := \frac1N\sum_{i=1}^N \xi_i\xi_i^\top .
		\]
		Then there exists an absolute constant $C>0$ (depending only on the subgaussian
		norm of $\xi_1$) such that for every $u\ge 0$,
		\begin{equation}\label{eq:conc}
			\Big\|\widehat\Sigma_N-\Sigma\Big\|_{\mathrm{op}}
			\le
			C\Bigg(\sqrt{\frac{M+u}{N}}+\frac{M+u}{N}\Bigg)\,\|\Sigma\|_{\mathrm{op}}
			\qquad\text{with probability at least }1-2e^{-u}.
		\end{equation}
		In particular, if $\Sigma=I_M$ (the $M$-dimensional identity matrix) and $M/N\to 0$, then for every fixed $\varepsilon>0$,
		\[
		\Pp\Big(\big\|\widehat\Sigma_N-I_M\big\|_{\mathrm{op}}\le \varepsilon\Big)\longrightarrow 1,
		\qquad\text{and hence}\qquad
		\Pp\Big(\|\widehat\Sigma_N\|_{\mathrm{op}}\le 1+\varepsilon\Big)\longrightarrow 1.
		\]
	\end{lemma}
	
	\begin{proof}
		The high-probability bound is exactly Exercise~4.7.3 in Vershynin
		\cite[Ex.~4.7.3]{VershyninHDP}, which follows from the covariance estimation theorem
		\cite[Thm.~4.7.1]{VershyninHDP}. (Vershynin's notation uses $m$ samples in $\mathbb R^n$;
		here $m=N$ and $n=M$.)
		
		Now assume $\Sigma=I_M$. Fix $\varepsilon>0$ and take $u_N:=\log N$ in \eqref{eq:conc}.
		Then, with probability at least $1-2e^{-u_N}=1-2/N$ we obtain,
		\[
		\big\|\widehat\Sigma_N-I_M\big\|_{\mathrm{op}}
		\le
		C\Bigg(\sqrt{\frac{M+1}{N}}+\frac{M+1}{N}\Bigg).
		\]
		Since $M/N\to 0$, the right-hand side converges to $0$, so
		$\|\widehat\Sigma_N-I_M\|_{\mathrm{op}}\to 0$ in probability.
		The final claim follows because
		$$\|\widehat\Sigma_N\|_{\mathrm{op}}\le \|I_M\|_{\mathrm{op}}+\|\widehat\Sigma_N-I_M\|_{\mathrm{op}}
		=1+\|\widehat\Sigma_N-I_M\|_{\mathrm{op}}.$$
	\end{proof}
	
	Together with Lemma~\ref{lem:TV}, this shows that in the high-temperature regime
	the total variation distance of the $k$-marginal is controlled by
	$
	\E_{Q_N}\|u\|\cdot \sqrt{k},
	$
	which will yield increasing propagation of chaos under the condition
	$kM/N\to0$.

	\section{Proof of Theorem~\ref{thm:highT}}
	
	\begin{proof}[Proof of Theorem~\ref{thm:highT}]
		Fix $\beta<1$ and as above choose $\varepsilon>0$ such that $\beta(1+\varepsilon)<1$ and let
		$$\mathcal E_N(\varepsilon):=\{\|\widehat\Sigma_N\|_{\op}\le 1+\varepsilon\}.$$
		By Lemma~\ref{lem:covconcentration}, since $M/N\to0$ we have
		\begin{equation}\label{eq:EN_high_prob}
			\P_\xi(\mathcal E_N(\varepsilon))\longrightarrow 1 .
		\end{equation}
		
		On $\mathcal E_N(\varepsilon)$, Lemma~\ref{lem:QN_domination} yields constants
		$C_0<\infty$ and $\alpha>0$ (depending only on $\beta,\varepsilon$) such that
		$Q_N\le C_0\gamma_\alpha$, where $\gamma_\alpha=\mathcal N(0,(2\alpha N)^{-1}I_M)$.
		In particular, still on $\mathcal E_N(\varepsilon)$ for each $i\le k$,
		\begin{equation}\label{eq:dir_var_bound}
			\E_{u\sim Q_N}\big[(u\cdot\xi_i)^2\big]
			\le
			C_0\,\E_{G\sim\gamma_\alpha}\big[(G\cdot\xi_i)^2\big]
			=
			C_0\,\frac{\|\xi_i\|^2}{2\alpha N}
			=
			C\,\frac{M}{N},
		\end{equation}
		since $\|\xi_i\|^2=\sum_{\nu=1}^M(\xi_i^\nu)^2=M$.
		
		By Lemma~\ref{lem:TV} and Jensen's inequality (again on $\mathcal E_N(\varepsilon)$),
		\begin{align}
			\TV(\mu_N^{(k)},\pi^{\otimes k})
			&\le
			\E_{u\sim Q_N}\Big[\sqrt{\sum_{i=1}^k (u\cdot\xi_i)^2}\Big]\\
			&\le
			\sqrt{\sum_{i=1}^k \E_{u\sim Q_N}(u\cdot\xi_i)^2}
			\notag\\
			&\le
			\sqrt{k\cdot C\frac{M}{N}}
			=
			C\,\sqrt{\frac{kM}{N}}
			\label{eq:TV_bound_final}
		\end{align}
		where we used \eqref{eq:dir_var_bound} in the last inequality.
		Since $kM/N\to0$, the right-hand side of \eqref{eq:TV_bound_final} tends to $0$.
		Together with \eqref{eq:EN_high_prob}, this implies
		$\TV(\mu_N^{(k)},\pi^{\otimes k})\to0$ in $\P_\xi$-probability.
	\end{proof}
	
	\begin{remark}
		\normalfont Under the high-temperature mixing measure $Q_N$, the Gaussian domination
		$Q_N\le C_0\mathcal N(0,(2\alpha N)^{-1}I_M)$ implies that typical $u$ has
		coordinates of size $N^{-1/2}$ and hence $\|u\|^2$ is of order $M/N$.
		For a fixed pattern $\xi_i\in\{-1,+1\}^M$, the local field
		$h_i=u\cdot\xi_i$ therefore has variance of order $M/N$, so $h_i$ is typically
		of size $\sqrt{M/N}$.
		The $k$-spin marginal under the mixture is close to product if the collection
		of means $\tanh(h_i)$ is small in $\ell^2$, and Lemma~\ref{lem:TV} bounds this
		distance by $\sqrt{\sum_{i=1}^k h_i^2}$.
		Since each $h_i^2$ has mean of order $M/N$, the natural size of the sum is
		$kM/N$, and propagation of chaos follows if $kM/N\to0$. This explains our condition $kM/N \to 0$.
	\end{remark}

	\section{When propagation of chaos stops for  $\beta<1$ --- Proof of Theorem \ref{thm:stop_macro_Msmall} }
	
	Let us next prove Theorem \ref{thm:stop_macro_Msmall}:
	
	\begin{proof}
		For the rest of the proof let us write $$\tau^2:=\frac{\beta}{1-\beta} \qquad \text {and set} \quad  \lambda_N:=\tau^2\frac{k}{N},$$ so that
		$\lambda_N\to \lambda:=\tau^2\rho>0$.
		For a given $\varepsilon>0$ let us introduce a disorder event similar to $\mathcal E_N$ in the previous sections. More precisely, define $\mathcal G_N$ to be the disorder event on which
		\[
		\mathcal{G}_N:=\left\{\xi=(\xi_i^\mu)_{i,\mu}:
		\Big\|\widehat\Sigma_N-I_M\Big\|_{\op}\le\varepsilon
		\text{  as well as }
		\Big\|\widehat\Sigma_k-I_M\Big\|_{\op}\le\varepsilon
		\right\}
		\]
		where $\widehat\Sigma_N=\frac1N\sum_{i=1}^N\xi_i\xi_i^\top$ and
		$\widehat\Sigma_k=\frac1k\sum_{i=1}^k\xi_i\xi_i^\top$.
		
		Since $M=o(N)$ and $k\sim\rho N$ (by assumption) also $M=o(k)$. Hence, Lemma~\ref{lem:covconcentration}
		applied with $N$ and with $k$ implies $$\P_\xi(\mathcal G_N)\to 1.$$
		
		\medskip
		Define the shorthand notation $\sigma_{i:j}:=(\sigma_i, \ldots, \sigma_j)$, (where we assume that $i<j$) and let $L_N$ be the density of 
		$\mu_N^{(k)}$ with respect to $\pi^{\otimes k}$
		(in statistics this is called the likelihood ratio): 
		\[
		L_N(\sigma_{1:k})
		:=
		\frac{d\mu_N^{(k)}}{d\pi^{\otimes k}}(\sigma_{1:k})
		=
		2^k\,\mu_N^{(k)}(\sigma_{1:k}).
		\]
		
		\begin{lemma}
			\label{lem:likelihood_k}
			For every $\sigma_{1:k}\in\{-1,+1\}^k$, one has the identity
			\begin{equation}\label{eq:LN_exact_lemma}
				L_N(\sigma_{1:k})
				=
				\frac{\displaystyle \int_{\mathbb R^M}
					\exp\!\Big(-\frac{N}{2\beta}\|u\|^2
					+\sum_{i=1}^k \sigma_i\,u\!\cdot\!\xi_i\Big)
					\prod_{i=k+1}^N \cosh(u\!\cdot\!\xi_i)\,du}
				{\displaystyle \int_{\mathbb R^M}
					\exp\!\Big(-\frac{N}{2\beta}\|u\|^2\Big)
					\prod_{i=1}^N \cosh(u\!\cdot\!\xi_i)\,du}.
			\end{equation}
			Equivalently, if $Q_N$ denotes the Hubbard--Stratonovich mixing measure
			from Lemma~\ref{lem:HS}, then
			\begin{equation}\label{eq:LN_expect_lemma}
				L_N(\sigma_{1:k})
				=
				\E_{Q_N}\Big[
				\exp\Big(\sum_{i=1}^k \big(\sigma_i\,u\!\cdot\!\xi_i
				- 
				\log\cosh(u\!\cdot\!\xi_i)\big)\Big)
				\Big].
			\end{equation}
		\end{lemma}
		
		\begin{proof}
			Fix $\sigma_{1:k}\in\{-1,+1\}^k$. By definition of the marginal,
			\[
			\mu_N^{(k)}(\sigma_{1:k})
			=
			\sum_{\sigma_{k+1:N}\in\{-1,+1\}^{N-k}}
			\mu_N(\sigma_{1:k},\sigma_{k+1:N})
			=
			\frac{1}{Z_N}
			\sum_{\sigma_{k+1:N}}
			\exp\Big(\frac{\beta N}{2}\|m_N(\sigma)\|^2\Big).
			\]
			Let us the apply Hubbard--Stratonovich identity (as in Lemma~\ref{lem:HS}) to the
			factor $\exp(\frac{\beta N}{2}\|m_N(\sigma)\|^2)$, obtaining
			\[
			\exp\Big(\frac{\beta N}{2}\|m_N(\sigma)\|^2\Big)
			=
			\Big(\frac{N}{2\pi\beta}\Big)^{M/2}
			\int_{\mathbb R^M}
			\exp\Big(-\frac{N}{2\beta}\|u\|^2 + \sum_{i=1}^N \sigma_i\,u\cdot\xi_i\Big)\,du .
			\]
			We insert this expression into the sum defining $\mu_N^{(k)}(\sigma_{1:k})$ and
			use Fubini to arrive at:
			\[
			\mu_N^{(k)}(\sigma_{1:k})
			=
			\frac{1}{Z_N}
			\Big(\frac{N}{2\pi\beta}\Big)^{M/2}
			\int_{\mathbb R^M}
			\exp\Big(-\frac{N}{2\beta}\|u\|^2 + \sum_{i=1}^k \sigma_i\,u\cdot\xi_i\Big)
			\sum_{\sigma_{k+1:N}}
			\exp\Big(\sum_{i=k+1}^N \sigma_i\,u\cdot\xi_i\Big)\,du.
			\]
			The remaining sum on the right hand side of the above equation factorizes:
			\[
			\sum_{\sigma_{k+1:N}}
			\exp\Big(\sum_{i=k+1}^N \sigma_i\,u\cdot\xi_i\Big)
			=
			\prod_{i=k+1}^N
			\sum_{\sigma_i=\pm1} e^{\sigma_i(u\cdot\xi_i)}
			=
			\prod_{i=k+1}^N 2\cosh(u\cdot\xi_i).
			\]
			Thus
			\[
			\mu_N^{(k)}(\sigma_{1:k})
			=
			\frac{1}{Z_N}
			\Big(\frac{N}{2\pi\beta}\Big)^{M/2}
			\int_{\mathbb R^M}
			\exp\Big(-\frac{N}{2\beta}\|u\|^2\Big)
			\exp\Big(\sum_{i=1}^k \sigma_i\,u\cdot\xi_i\Big)
			\prod_{i=k+1}^N 2\cosh(u\cdot\xi_i)\,du.
			\]
			On the other hand, applying the same Hubbard--Stratonovich computation to
			the full partition function gives as in Lemma \ref{lem:HS}
			\[
			Z_N
			=
			\Big(\frac{N}{2\pi\beta}\Big)^{M/2}
			\int_{\mathbb R^M}
			\exp\Big(-\frac{N}{2\beta}\|u\|^2\Big)
			\prod_{i=1}^N 2\cosh(u\cdot\xi_i)\,du.
			\]
			Dividing the two displays, canceling the common prefactor
			$(\frac{N}{2\pi\beta})^{M/2}$, and multiplying numerator and denominator by $2^{-k}$
			yields \eqref{eq:LN_exact_lemma}. Finally, writing
			\[
			\exp\Big(\sum_{i=1}^k \sigma_i\,u\cdot\xi_i\Big)
			=
			\exp\Big(\sum_{i=1}^k \big(\sigma_i\,u\cdot\xi_i - 
			\log\cosh(u\cdot\xi_i)\big)\Big)
			\prod_{i=1}^k \cosh(u\cdot\xi_i),
			\]
			and recognizing the resulting density as $Q_N(du)$ (up to normalization) gives
			\eqref{eq:LN_expect_lemma}.
		\end{proof}
		
		Hence,
		$
		L_N(\sigma_{1:k})
		=
		\E_{Q_N}\Big[
		\exp\Big(\sum_{i=1}^k \big(\sigma_i\,u\cdot\xi_i-\log\cosh(u\cdot\xi_i)\big)\Big)
		\Big].
		$
		Define $$S_{N,k}:=\frac{1}{\sqrt k}\sum_{i=1}^k\sigma_i\xi_i\in\R^M,$$ so that
		$$\sum_{i=1}^k \sigma_i\,u\cdot\xi_i=\sqrt k\,u\cdot S_{N,k}.$$
		
		\medskip
		
		For the next step we need the following corollary of our Gaussian domination result for $Q_N$
		
		\begin{corollary}
			\label{lem:proj_moments}
			Recall that by Lemma \ref{lem:QN_domination} on $\mathcal E_N(\varepsilon)$ there are $C_0>0$ and
			$\alpha>0$ such that we have the domination
			$$Q_N\le C_0\gamma_\alpha.$$ 
			Here $\gamma_\alpha=\mathcal N(0,(2\alpha N)^{-1}I_M)$.
			Then for every deterministic vector $w\in\R^M$,
			\[
			\E_{u\sim Q_N}\big[(u\cdot w)^2\big]
			\le
			C\,\frac{\|w\|^2}{N} \qquad \text{as well as }\quad
			\E_{u\sim Q_N}\big[(u\cdot w)^4\big]
			\le
			C\,\frac{\|w\|^4}{N^2},
			\]
			where $C<\infty$ depends only on $(C_0,\alpha)$.
			In particular, since $\|\xi_i\|^2=M$,
			\[
			\E_{u\sim Q_N}\big[(u\cdot\xi_i)^2\big]\le C\frac{M}{N} \qquad \text{as well as }\quad
			\E_{u\sim Q_N}\big[(u\cdot\xi_i)^4\big]\le C\frac{M^2}{N^2}.
			\]
		\end{corollary}
		
		\begin{proof}
			From the above arguments we deduce that, 
			on $\mathcal G_N$, the density of $Q_N$ can be written as
			\[
			\frac{dQ_N}{du}(u)
			=
			\frac{1}{Z_N}
			\exp\!\Big(-\frac{N}{2\tau^2}\|u\|^2\Big)
			\exp\!\big(\Delta_N(u)\big),
			\]
			where
			\[
			\Delta_N(u):=\sum_{i=1}^N R(u\cdot\xi_i),
			\qquad |R(t)|\le C|t|^4.
			\]
			Equivalently,
			\begin{equation}\label{eq:RN_QN_GN}
				\frac{dQ_N}{dG_N}(u)
				=
				\frac{\exp(\Delta_N(u))}{\E_{G_N}[\exp(\Delta_N(U))]},
			\end{equation}
			where $G_N=\mathcal N(0,\tau^2 N^{-1} I_M)$.
			
			\medskip
			We first control $\Delta_N$ under the Gaussian measure $G_N$.
			For $U\sim G_N$, $U\cdot\xi_i$ is centered Gaussian with variance
			$\tau^2\|\xi_i\|^2/N=\tau^2 M/N$, hence
			$
			\E_{G_N}[(U\cdot\xi_i)^4]
			=
			3\Big(\frac{\tau^2 M}{N}\Big)^2.$
			
			Using $|R(t)|\le C|t|^4$, we obtain
			\[
			\E_{G_N}[|\Delta_N(U)|]
			\le
			C\sum_{i=1}^N \E_{G_N}[(U\cdot\xi_i)^4]
			\le
			C\,\frac{M^2}{N}.
			\]
			Since $M=o(\sqrt N)$, it follows by Markov's inequality that
			\[
			\Delta_N(U)\xrightarrow[N\to\infty]{}0
			\quad\text{in $G_N$-probability on $\mathcal G_N$.}
			\]
			
			\medskip
			Fix $R>0$ and define
			\[
			B_{N,R}:=\Big\{\|u\|\le R\sqrt{\frac{M}{N}}\Big\}.
			\]
			By Gaussian concentration,
			\(
			G_N(B_{N,R}^c)\to0
			\)
			as $R\to\infty$, uniformly in $N$.
			On $B_{N,R}$, we have $\Delta_N(u)=o(1)$ uniformly in $u$ on $\mathcal G_N$,
			hence $\exp(\Delta_N(u))=1+o(1)$ uniformly on $B_{N,R}$.
			
			\medskip
			Now let
			$
			f(u):=\exp\!\Big(\sqrt{k}\,u\cdot s_N-\frac{k}{2}\|u\|^2\Big),
			$
			with $\|s_N\|=O(\sqrt M)$.
			Using \eqref{eq:RN_QN_GN}, we write
			\[
			\E_{Q_N}[f]
			=
			\frac{\E_{G_N}[f(U)\exp(\Delta_N(U))]}
			{\E_{G_N}[\exp(\Delta_N(U))]}.
			\]
			The contributions from $B_{N,R}^c$ vanish uniformly by Gaussian tails,
			while on $B_{N,R}$ we may replace $\exp(\Delta_N(U))$ by $1+o(1)$.
			Since $\E_{G_N}[\exp(\Delta_N(U))]\to1$, this yields
			\[
			\E_{Q_N}[f]
			=
			\E_{G_N}[f]\,(1+o(1))
			\]
			in $\P_\xi$-probability.
			This proves the claim.
		\end{proof}

		Next, we use Taylor expansion of $\log \cosh$ to 
		second order: 
		$\log\cosh t=\frac{t^2}{2}+R(t)$ with $|R(t)|\le C|t|^4$.
		On $\mathcal G_N$,
		\[
		\sum_{i=1}^k (u\cdot\xi_i)^2
		=
		k\,u^\top\widehat\Sigma_k u
		=
		k\|u\|^2+O(\varepsilon)\,k\|u\|^2.
		\]
		Moreover,
		$
		\sum_{i=1}^k |R(u\cdot\xi_i)|
		\le
		C\sum_{i=1}^k (u\cdot\xi_i)^4.
		$
		
		Consequently, on $\mathcal G_N$,
		\begin{equation}\label{eq:LN_quad_Mgrowing_withR}	
			L_N(\sigma_{1:k})
			=
			\E_{Q_N}\Big[\exp(\sqrt{k}\,u\cdot S_{N,k}-\tfrac12\sum_{i=1}^k (u\cdot\xi_i)^2)\,
			\exp(-\sum_{i=1}^k R(u\cdot\xi_i))\Big].
		\end{equation}
		
		We can now remove the Taylor remainder.
		
		\begin{lemma}\label{lem:remove_R}
			Assume $k/N\to\rho\in(0,1)$ and $M=o(\sqrt N)$. On $\mathcal G_N$,
			\[
			\E_{Q_N}\Big[\Big|\exp\!\Big(-\sum_{i=1}^k R(u\cdot\xi_i)\Big)-1\Big|\Big]
			\longrightarrow 0
			\qquad\text{in }\mathbb P_\xi\text{-probability}.
			\]
		\end{lemma}

		\begin{proof}
			Fix $R>0$ and define the ball $
			B_{N,R}$ as above.
			By Gaussian domination of $Q_N$ (Lemma~\ref{lem:QN_domination}), we have
			$\sup_N Q_N(B_{N,R}^c)\to0$ as $R\to\infty$, in $\mathbb P_\xi$-probability.
			
			Write $\Delta_k(u):=\sum_{i=1}^k R(u\cdot\xi_i)$. Using
			$|e^{-x}-1|\le e^{|x|}\,|x|$, we obtain
			\[
			\E_{Q_N}\Big[|e^{-\Delta_k(u)}-1|\Big]
			\le
			\E_{Q_N}\Big[e^{|\Delta_k(u)|}\,|\Delta_k(u)|\,\mathbf 1_{B_{N,R}}\Big]
			+
			2\,Q_N(B_{N,R}^c).
			\]
			We treat the two terms separately:
			
			a) On $\mathcal G_N$ and $B_{N,R}$ we have $\|u\|^2\lesssim M/N$, hence
			\[
			\sum_{i=1}^k (u\cdot\xi_i)^2
			=
			k\,u^\top\widehat\Sigma_k u
			\le
			k(1+\varepsilon)\|u\|^2
			\lesssim
			k\,\frac{M}{N}
			\asymp
			M,
			\]
			so in particular $|u\cdot\xi_i|\le \|u\|\,\|\xi_i\|\lesssim R\,\frac{M}{\sqrt N}$.
			Since $M=o(\sqrt N)$, this bound tends to $0$ uniformly on $B_{N,R}$.
			Using $|R(t)|\le C|t|^4$, we get for $N$ large (depending on $R$),
			\[
			|\Delta_k(u)|
			\le
			C\sum_{i=1}^k |u\cdot\xi_i|^4
			\le
			C\Big(\max_{i\le k}|u\cdot\xi_i|^2\Big)\sum_{i=1}^k (u\cdot\xi_i)^2
			=o(1)\cdot O(M)
			=o(1),
			\]
			uniformly on $B_{N,R}$.
			Hence $e^{|\Delta_k(u)|}\le 2$ for $N$ large, and therefore
			\[
			\E_{Q_N}\Big[e^{|\Delta_k(u)|}\,|\Delta_k(u)|\,\mathbf 1_{B_{N,R}}\Big]
			\le
			2\,\E_{Q_N}\big[|\Delta_k(u)|\big].
			\]
			Finally, using again $|R(t)|\le C|t|^4$ and Corollary~\ref{lem:proj_moments},
			\[
			\E_{Q_N}\big[|\Delta_k(u)|\big]
			\le
			C\sum_{i=1}^k \E_{Q_N}[(u\cdot\xi_i)^4]
			\le
			C\,k\,\frac{M^2}{N^2}
			\sim
			C\,\rho\,\frac{M^2}{N}
			\longrightarrow 0.
			\]
			
			\medskip
			b) For the tail term we
			choose $R$ large so that $Q_N(B_{N,R}^c)\le \delta$ with high probability,
			and then let $N\to\infty$.
			Combining Step~a) and Step~b) yields the claim.
		\end{proof}
		
		Therefore, on $\mathcal G_N$
		\begin{equation}\label{eq:LN_quad_Mgrowing}
			L_N(\sigma_{1:k})
			=
			\E_{Q_N}\Big[\exp\Big(\sqrt k\,u\cdot S_{N,k}-\frac{k}{2}\|u\|^2\Big)\Big]+r_N
		\end{equation}
		where $\sup_{\sigma_{1:k}}|r_N|\to0$ in $\mathbb P_\xi$-probability. 
		
		\medskip
		In the next step we apply again the Taylor expansion
		\[
		\log\cosh t=\frac{t^2}{2}+R(t), \qquad |R(t)|\le C|t|^4,
		\]
		this time to the full sum $\sum_{i=1}^N \log\cosh(u\cdot\xi_i)$
		appearing in the definition of the mixing measure $Q_N$ given by
		\[
		Q_N(du)\propto
		\exp\Big(-\frac{N}{2\beta}\|u\|^2\Big)\prod_{i=1}^N \cosh(u\cdot\xi_i)\,du.
		\]
		By the same arguments as in the previous step, on $\mathcal G_N$, this yields
		\[
		\sum_{i=1}^N\log\cosh(u\cdot\xi_i)
		=
		\frac12\sum_{i=1}^N (u\cdot\xi_i)^2 + \sum_{i=1}^N R(u\cdot\xi_i)
		=
		\frac{N}{2}\|u\|^2 + \Delta_N(u),
		\]
		where $\Delta_N(u)=\sum_{i=1}^N R(u\cdot\xi_i)$ satisfies
		\[
		\E_{Q_N}[|\Delta_N(u)|]\le C\frac{M^2}{N}\to0
		\quad\text{on }\mathcal G_N.
		\]

		The following lemma allows to replace expectation with respect to $Q_N$ by expectation with respect to an appropriate Gaussian measure for an
		expoential intergal we need. While Corollary~\ref{lem:proj_moments} provides moment bounds under $Q_N$,
		the next lemma shows that, for the exponential observables appearing in the
		likelihood ratio, the mixing measure $Q_N$ may be asymptotically replaced by a Gaussian.

		\begin{lemma}
			\label{lem:QN_gauss_Mgrowing}
			Fix $\beta<1$ and assume $M=o(\sqrt N)$.
			Let $\tau^2=\frac{\beta}{1-\beta}$ and $G_N=\mathcal N(0,\tau^2N^{-1}I_M)$.
			Then on the disorder event $\mathcal G_N$,
			for every sequence $s_N\in\R^M$ with $\|s_N\|=O(\sqrt M)$ and
			every $k=k(N)$ with $k/N\to\rho\in(0,1)$,
			\[
			\E_{Q_N}\!\left[
			e^{\sqrt{k}\,u\cdot s_N-\frac{k}{2}\|u\|^2}
			\right]
			=
			\E_{G_N}\!\left[
			e^{\sqrt{k}\,u\cdot s_N-\frac{k}{2}\|u\|^2}
			\right]\,(1+o(1)),
			\]
			where $o(1)\to0$ in $\P_\xi$-probability.
		\end{lemma}
		
		\begin{proof}
			From the above arguments we deduce that, 
			on $\mathcal E_N(\varepsilon)$, the density of $Q_N$ can be written as
			\[
			\frac{dQ_N}{du}(u)
			=
			\frac{1}{Z_N}
			\exp\!\Big(-\frac{N}{2\tau^2}\|u\|^2\Big)
			\exp\!\Big(\Delta_N(u)\Big),
			\]
			where
			$
			\Delta_N(u):=\sum_{i=1}^N R(u\cdot\xi_i).
			$
			Equivalently,
			\begin{equation}\label{eq:RN_QN_GN}
				\frac{dQ_N}{dG_N}(u)
				=
				\frac{\exp(\Delta_N(u))}{\E_{G_N}[\exp(\Delta_N(U))]},
			\end{equation}
			where $G_N=\mathcal N(0,\tau^2 N^{-1} I_M)$.
			
			\medskip
			As seen above	
			$
			\E_{Q_N}[|\Delta_N(u)|]\longrightarrow 0
			\quad\text{on }\mathcal G_N.
			$
			
			\medskip
			
			Fix $R>0$ and define 
			$B_{N,R}:=\Big\{\|u\|\le R\sqrt{\frac{M}{N}}\Big\}
			$ as above.
			By Gaussian concentration,
			$
			\sup_N G_N(B_{N,R}^c)\xrightarrow[R\to\infty]{}0,
			$
			and by Gaussian domination of $Q_N$ the same holds for $Q_N(B_{N,R}^c)$
			in $\P_\xi$-probability.

			Again let
			$
			f(u):=\exp\!\Big(\sqrt k\,u\cdot s_N-\frac{k}{2}\|u\|^2\Big).
			$
			Using \eqref{eq:RN_QN_GN} we may write
			\[
			\E_{Q_N}[f]
			=
			\frac{\E_{G_N}[f(U)\exp(\Delta_N(U))]}
			{\E_{G_N}[\exp(\Delta_N(U))]}.
			\]
			On $B_{N,R}$, $\exp(\Delta_N(U))=1+o(1)$ uniformly, while the contributions
			from $B_{N,R}^c$ vanish uniformly in $N$ by the tail bounds.
			Hence,
			\[
			\E_{Q_N}[f]
			=
			\E_{G_N}[f]\,+o(1)
			\]
			in $\P_\xi$-probability.
			This proves the claim.
		\end{proof}

		
		Now, on $\mathcal G_N$, Lemma~\ref{lem:QN_gauss_Mgrowing} allows us to replace $Q_N$ in
		\eqref{eq:LN_quad_Mgrowing} by $G_N=\mathcal N(0,\tau^2N^{-1}I_M)$ at a multiplicative
		$(1+o(1))$ cost, uniformly in $\sigma_{1:k}$. With $a=\sqrt{k}\,S_{N,k}$,
		$B=kI_M$ and $\Sigma=\tau^2N^{-1}I_M$, the Gaussian identity
		\[
		\E\exp\!\Big(a\cdot u-\tfrac12 u^\top Bu\Big)
		=\det(I+\Sigma B)^{-1/2}\exp\!\Big(\tfrac12 a^\top(\Sigma^{-1}+B)^{-1}a\Big)
		\]
		yields $\det(I+\Sigma B)^{-1/2}=(1+\lambda_N)^{-M/2}$ and
		$\tfrac12 a^\top(\Sigma^{-1}+B)^{-1}a=\frac{\lambda_N}{2(1+\lambda_N)}\|S_{N,k}\|^2$,
		where $\lambda_N=\tau^2 k/N$, giving 	
		\begin{equation}\label{eq:LN_form_Mgrowing}
			L_N(\sigma_{1:k})
			=
			(1+\lambda_N)^{-M/2}
			\exp\Big(\frac{\lambda_N}{2(1+\lambda_N)}\|S_{N,k}\|^2\Big)
			+ r_N,
		\end{equation}
		where
		$\sup_{\sigma_{1:k}}|r_N|\to0
		\quad\text{in }\mathbb P_\xi\text{-probability}.$
		
		\medskip
		Note that under $\pi^{\otimes k}$ and on $\mathcal G_N$,
		\[
		\|S_{N,k}\|^2
		=
		\frac{1}{k}\sigma^\top(\Xi\Xi^\top)\sigma,
		\qquad \Xi=(\xi_1^\top,\dots,\xi_k^\top),
		\]
		is a quadratic form in a Rademacher vector.
		By the Hanson--Wright inequality (see \cite[Theorem 6.2.1]{VershyninHDP}) and on the event $\mathcal G_N$,
		\[
		\frac{1}{M}\|S_{N,k}\|^2 \xrightarrow{\ \pi^{\otimes k}\ } 1
		\qquad\text{provided } M=o(k) \text{  (which is true in our situation)}.
		\]
		
		Indeed, write
		\[
		S_{N,k}=\frac1{\sqrt{k}}\Xi^\top\sigma,
		\qquad
		\|S_{N,k}\|^2=\frac1{k}\sigma^\top(\Xi\Xi^\top)\sigma,
		\]
		where $\Xi$ is the $k\times M$ matrix with rows $\xi_i^\top$.
		Setting $A:=k^{-1}\Xi\Xi^\top$, we have
		$\|S_{N,k}\|^2=\sigma^\top A\sigma$ and
		$\E_{\pi^{\otimes k}}[\sigma^\top A\sigma]=\tr(A)=M$.
		On $\mathcal G_N$ one has
		$\|A\|_{\op}=\|\widehat\Sigma_k\|_{\op}\le 1+\varepsilon$
		and $\|A\|_F^2=\tr(A^2)\le M(1+\varepsilon)$ since $$\tr(A^2)=\tr((\widehat\Sigma_k)^2)\le \mathrm{rank}(\widehat\Sigma_k)\|\widehat\Sigma_k\|_{\op}^2
		\le M(1+\varepsilon)^2.$$
		Indeed,
		\[
		\tr\left(\left(\widehat{\Sigma}_k\right)^2\right) \leq \mathrm{rank}\left(\widehat{\Sigma}_k\right) \left\|\widehat{\Sigma}_k\right\|_{\text{op}}^2.
		\]
		In our setting $\widehat{\Sigma}_k$ is an $M \times M$ matrix, so $\mathrm{rank}\left(\widehat{\Sigma}_k\right) \leq M$, hence
		\[
		\tr\left(\left(\widehat{\Sigma}_k\right)^2\right) \leq M \left\|\widehat{\Sigma}_k\right\|_{\text{op}}^2 \leq M(1+\varepsilon)^2 \quad \text{on } \mathcal{G}_N.
		\]
		
		Hence, by the Hanson--Wright inequality
		\cite[Theorem~6.2.1]{VershyninHDP}, for every $\delta>0$,
		\begin{equation}\label{eq:cons_Hanson_wright}
			\pi^{\otimes k}\!\left(
			\left|\frac1M\|S_{N,k}\|^2-1\right|\ge\delta
			\right)
			\le
			2\exp(-c\,\delta^2 M),
		\end{equation}
		which tends to $0$ as $M\to\infty$.
		
		Since $L_N$ depends on $\sigma$ only through $\|S_{N,k}\|^2$
		via \eqref{eq:LN_form_Mgrowing}, it suffices to control the
		fluctuations of $\|S_{N,k}\|^2$ under $\pi^{\otimes k}$.
		By \eqref{eq:cons_Hanson_wright}, on $\mathcal G_N$ we have
		\[
		\frac{1}{M}\|S_{N,k}\|^2 \xrightarrow{\ \pi^{\otimes k}\ } 1.
		\]
		
		Recall that on $\mathcal G_N$,
		\[
		L_N(\sigma_{1:k})
		=
		(1+\lambda_N)^{-M/2}
		\exp\!\Big(
		\frac{\lambda_N}{2(1+\lambda_N)}\|S_{N,k}\|^2
		\Big)
		+ r_N,
		\qquad
		\sup_{\sigma_{1:k}}|r_N|\to 0.
		\]
		
		Let $Z_N\sim\mathcal N(0,I_M)$ and define the Gaussian proxy
		\[
		\widetilde L_N
		:=
		(1+\lambda_N)^{-M/2}
		\exp\!\Big(
		\frac{\lambda_N}{2(1+\lambda_N)}\|Z_N\|^2
		\Big),
		\]
		which is the likelihood ratio
		\[
		\frac{d\mathcal N(0,(1+\lambda_N)I_M)}
		{d\mathcal N(0,I_M)}(Z_N).
		\]
		
		By the concentration of $\|S_{N,k}\|^2$ and
		Lemma~\ref{lem:UI_LN_correct} (uniform integrability),
		we obtain
		\[
		\E_{\pi^{\otimes k}}|L_N-1|
		=
		\E|\widetilde L_N-1|
		+ o(1)
		\qquad\text{in }\mathbb P_\xi\text{-probability}.
		\]
		
		Therefore,
		\begin{equation} \label{eq:TV_LN}
			\TV(\mu_N^{(k)},\pi^{\otimes k})
			=
			\frac12\,\E_{\pi^{\otimes k}}|L_N-1|
			\longrightarrow
			\TV\big(
			\mathcal N(0,I_M),
			\mathcal N(0,(1+\lambda)I_M)
			\big)
		\end{equation}
		in $\mathbb P_\xi$-probability.
		
		As intermediate step we need to 
		check uniform integrability of $(L_N)_N$ under $\pi^{\otimes k}$ on $\mathcal G_N$:
		
		\begin{lemma}
			\label{lem:UI_LN_correct}
			Assume $k/N\to\rho\in(0,1)$ and $M=o(k)$, and set
			$\lambda_N:=\tau^2 k/N$ with $\tau^2=\beta/(1-\beta)$, so $\lambda_N\to\lambda>0$.
			Let $L_N$ be given on $\mathcal G_N$ by
			\[
			L_N
			=
			(1+\lambda_N)^{-M/2}
			\exp\!\Big(\alpha_N\|S_{N,k}\|^2\Big),
			\qquad \text{with}\quad
			\alpha_N:=\frac{\lambda_N}{2(1+\lambda_N)}.
			\]
			Then there exists $\eta>0$ such that
			\[
			\sup_N \E_{\pi^{\otimes k}}\big[L_N^{1+\eta}\mathbf 1_{\mathcal G_N}\big] <\infty.
			\]
			In particular, $(L_N)_N$ is uniformly integrable under $\pi^{\otimes k}$ on $\mathcal G_N$.
		\end{lemma}
		
		\begin{proof}
			Fix $\eta>0$ and define $t_N:=(1+\eta)\alpha_N$.
			Then on $\mathcal G_N$,
			\[
			L_N^{1+\eta}
			=
			(1+\lambda_N)^{-(1+\eta)M/2}\,
			\exp\!\big(t_N\|S_{N,k}\|^2\big).
			\]
			We claim that for every fixed $t<1/2$ there exists $C_t<\infty$ such that
			\begin{equation}\label{eq:mgf_bound_SN}
				\E_{\pi^{\otimes k}}\!\left[\exp\!\big(t\|S_{N,k}\|^2\big)\mathbf 1_{\mathcal G_N}\right]
				\le
				C_t\,(1-2t)^{-M/2}.
			\end{equation}
			This is a standard moment generating function bound for Rademacher quadratic forms with bounded operator norm;
			we record a self-contained statement as Lemma~\ref{lem:mgf_quad} in the appendix.
			
			Assuming \eqref{eq:mgf_bound_SN} for the moment and taking $t=t_N$, we obtain
			\[
			\E_{\pi^{\otimes k}}\big[L_N^{1+\eta}\mathbf 1_{\mathcal G_N}\big]
			\le
			C\,
			(1+\lambda_N)^{-(1+\eta)M/2}\,(1-2t_N)^{-M/2}.
			\]
			Since $t_N\to t_*:=(1+\eta)\frac{\lambda}{2(1+\lambda)}<1/2$ for $\eta$ small,
			we have $1-2t_N\to 1-2t_*>0$.
			Moreover,
			\[
			\frac{1}{1-2t_N}
			\longrightarrow
			\frac{1}{1-2t_*}
			<
			(1+\lambda)^{1+\eta}
			=
			\lim_{N\to\infty}(1+\lambda_N)^{1+\eta}.
			\]
			Hence for all $N$ large enough,
			$(1-2t_N)^{-1}\le (1+\lambda_N)^{1+\eta/2}$,
			and therefore
			\[
			(1+\lambda_N)^{-(1+\eta)M/2}\,(1-2t_N)^{-M/2}
			\le
			(1+\lambda_N)^{-\eta M/4}\le 1.
			\]
			This yields $\sup_N \E[L_N^{1+\eta}\mathbf 1_{\mathcal G_N}]<\infty$.
			
			It remains to prove \eqref{eq:mgf_bound_SN}.
			To this end, let us abuse notation and write $\sigma=\sigma_{1:k}=(\sigma_1,\dots,\sigma_k)$ and let $\Xi$ be the $k\times M$ matrix
			with rows $\xi_i^\top$. Then
			\[
			S_{N,k}=\frac{1}{\sqrt{k}}\Xi^\top\sigma,
			\qquad
			\|S_{N,k}\|^2=\frac{1}{k}\sigma^\top(\Xi\Xi^\top)\sigma.
			\]
			On $\mathcal G_N$ one has $\|\widehat\Sigma_k\|_{\op}\le 1+\varepsilon$, hence
			\[
			\|\Xi\Xi^\top\|_{\op}=\|\Xi^\top\Xi\|_{\op}=k\|\widehat\Sigma_k\|_{\op}\le k(1+\varepsilon).
			\]
			Therefore $\|S_{N,k}\|^2=\sigma^\top A\sigma$ is a Rademacher quadratic form with
			$\|A\|_{\op}\le 1+\varepsilon$ on $\mathcal G_N$. Applying Lemma~\ref{lem:mgf_quad}
			yields \eqref{eq:mgf_bound_SN}.
		\end{proof}
		
		Jumping back to \eqref{eq:TV_LN} we have
		\[
		\TV(\mu_N^{(k)},\pi^{\otimes k})
		=
		\frac12\,\E_{\pi^{\otimes k}}|L_N-1|
		\longrightarrow
		\TV\big(\mathcal N(0,I_M),\mathcal N(0,(1+\lambda)I_M)\big)
		\]
		in $\P_\xi$-probability.
		
		\medskip
		
		To finish the proof, let $X\sim\mathcal N(0,I_M)$ and $Y\sim\mathcal N(0,(1+\lambda)I_M)$.
		Then $\|X\|^2\sim\chi^2_M$ and $\|Y\|^2\sim(1+\lambda)\chi^2_M$.
		
		Let $r_\ast^2:=\frac{M(1+\lambda)\log(1+\lambda)}{\lambda}$be the unique radius where the two densities cross.
		Define the (optimal) set is $$A_\ast:=\{\|x\|\ge r_\ast\}.$$ Then,
		\[
		\TV(\mathcal N(0,I_M),\mathcal N(0,(1+\lambda)I_M))
		=
		\P(Y\in A_\ast)-\P(X\in A_\ast).
		\]
		Since $\chi^2_M/M\to1$ in probability and concentrates at scale $\sqrt M$,
		one has $\|X\|^2/M\to1$ while $\|Y\|^2/M\to 1+\lambda$,
		and $r_\ast^2/M$ lies strictly between $1$ and $1+\lambda$.
		Therefore $\P(X\in A_\ast)\to0$ and $\P(Y\in A_\ast)\to1$, hence
		\[
		\TV(\mathcal N(0,I_M),\mathcal N(0,(1+\lambda)I_M))\longrightarrow 1
		\qquad (M\to\infty).
		\]
		Combining the previous steps yields $\TV(\mu_N^{(k)},\pi^{\otimes k})\to1$ in $\P_\xi$-probability.
	\end{proof}

	\section{ Proof of Theorem \ref{thm:crit_poc}}

	\subsection{Hubbard--Stratonovich representation at $\beta=1$}
	
	We start the proof of Theorem \ref{thm:crit_poc}
	with some considerations that are also useful for the 
	Proof of Theorem \ref{thm:crit_break_sqrtN}.
	
	\begin{lemma}
		\label{lem:HS_likelihood_beta1}
		Let $\beta=1$. Define the Hubbard--Stratonovich mixing measure $\widetilde Q_{N,1}$ on
		$\R^M$ by
		\begin{equation}\label{eq:tildeQN}
			\widetilde Q_{N,1}(du)
			=
			\frac{1}{\widetilde Z_{N,1}}\,
			\exp\!\Big(-\frac{N}{2}\|u\|^2\Big)
			\prod_{i=1}^N \cosh(u\cdot\xi_i)\,du,
		\end{equation}
		
		where $\widetilde Z_{N,1}$ is the normalizing constant making $\widetilde Q_{N,1}$ a probability measure.
		
		Then the likelihood ratio of the $k$-marginal w.r.t.\ $\pi^{\otimes k}$ is
		\[
		L_N(\sigma_{1:k})
		=
		\frac{d\mu_N^{(k)}}{d\pi^{\otimes k}}(\sigma_{1:k})
		=
		\E_{u\sim \widetilde Q_{N,1}}
		\exp\Big(\sum_{i=1}^k\big(\sigma_i\,u\cdot\xi_i-\log\cosh(u\cdot\xi_i)\big)\Big).
		\]
	\end{lemma}
	\begin{proof}
		Fix $\beta=1$ and $k \le N$ and
		consider $\mu_N^{(k)}$. The usual
		Hubbard-Stratonovich transformation gives for $\sigma_{1:k}\in\{\pm1\}^k$,
		\begin{align*}
			\mu_N^{(k)}(\sigma_{1:k})
			&=
			\sum_{\sigma_{k+1:N}}\mu_N(\sigma) \\
			&=
			\frac{1}{Z_N}\sum_{\sigma_{k+1:N}}
			\Big(\frac{N}{2\pi}\Big)^{M/2}\int_{\R^M}
			\exp\!\Big(-\frac{N}{2}\|u\|^2 + \sum_{i=1}^N \sigma_i\,u\cdot\xi_i\Big)\,du \\
			&=
			\frac{1}{Z_N}\Big(\frac{N}{2\pi}\Big)^{M/2}\int_{\R^M}
			\exp\!\Big(-\frac{N}{2}\|u\|^2 + \sum_{i=1}^k \sigma_i\,u\cdot\xi_i\Big)
			\prod_{i=k+1}^N 2\cosh(u\cdot\xi_i)\,du.
		\end{align*}
		
		Dividing by by $\pi^{\otimes k}(\sigma_{1:k})=2^{-k}$ we obtain the likelihood ratio
		\begin{align*}
			L_N(\sigma_{1:k})
			&=\frac{d\mu_N^{(k)}}{d\pi^{\otimes k}}(\sigma_{1:k})\\
			&=
			\frac{1}{\widetilde Z_{N,1}}
			\int_{\R^M}
			\exp\!\Big(-\frac{N}{2}\|u\|^2\Big)
			\prod_{i=1}^N \cosh(u\cdot\xi_i)\,
			\exp\!\Big(\sum_{i=1}^k (\sigma_i\,u\cdot\xi_i-\log\cosh(u\cdot\xi_i))\Big)\,du \\
			&=
			\E_{u\sim \widetilde Q_{N,1}}
			\exp\Big(\sum_{i=1}^k\big(\sigma_i\,u\cdot\xi_i-\log\cosh(u\cdot\xi_i)\big)\Big),
		\end{align*}
		where $\widetilde Q_{N,1}$ is the probability measure from \eqref{eq:tildeQN}.
	\end{proof}

	From here we conclude:
	\begin{corollary}
		\label{cor:second_moment_exact_beta1}
		Let $u,u'$ be i.i.d.\ with law $\widetilde Q_{N,1}$. Then
		\[
		\E_{\pi^{\otimes k}}[L_N^2]
		=
		\E_{u,u'\sim \widetilde Q_{N,1}}
		\prod_{i=1}^k
		\frac{\cosh((u+u')\cdot\xi_i)}{\cosh(u\cdot\xi_i)\cosh(u'\cdot\xi_i)}.
		\]
	\end{corollary}
	
	\begin{proof}
		By Lemma~\ref{lem:HS_likelihood_beta1},
		$
		L_N(\sigma_{1:k})
		=
		\E_{u\sim \widetilde Q_{N,1}}
		\exp\Big(\sum_{i=1}^k(\sigma_i\,u\cdot\xi_i-\log\cosh(u\cdot\xi_i))\Big).
		$
		Hence, writing $u,u'$ for i.i.d.\ samples from $\widetilde Q_{N,1}$,
		\[
		L_N(\sigma_{1:k})^2
		=
		\E_{u,u'}\exp\Big(\sum_{i=1}^k\big(\sigma_i (u+u')\cdot\xi_i-\log\cosh(u\cdot\xi_i)-\log\cosh(u'\cdot\xi_i)\big)\Big).
		\]
		Taking expectation over $\sigma\sim\pi^{\otimes k}$ and using
		$\E_{\pi}[e^{\sigma t}]=\cosh(t)$ yields
		\[
		\E_{\pi^{\otimes k}}[L_N^2]
		=
		\E_{u,u'}
		\prod_{i=1}^k
		\frac{\cosh((u+u')\cdot\xi_i)}{\cosh(u\cdot\xi_i)\cosh(u'\cdot\xi_i)} ,
		\]
		which is the desired identity.
	\end{proof}
	
	\medskip
	
	New will next need a Lemma which a direct consequence of the result and techniques in \cite{gentzloewe}.
	\begin{lemma}
		\label{lem:HS_quartic}
		There exist constants $a,C<\infty$ and a disorder event $\mathcal H_N$
		with $\mathbb P_\xi(\mathcal H_N)\to 1$ such that
		\[
		\E_{u\sim \widetilde Q_{N,1}}\Big[\exp\big(a\|N^{1/4}u\|^4\big)\mathbf 1_{\mathcal H_N}\Big]\le C
		\qquad\text{for all }N.
		\]
	\end{lemma}
	
	\begin{proof}
		Let $x:=N^{1/4}u$. By definition of $\widetilde Q_{N,1}$,
		the law of $x$ under $\widetilde Q_{N,1}$ has a density (w.r.t.\ Lebesgue measure) of the form
		\begin{equation}\label{eq:HS_density_x}
			\frac{1}{\tilde Z_N}\exp\!\Big(-N\,\Phi_N(x/N^{1/4})\Big),
			\qquad \text{where }\quad
			\Phi_N(v):=\frac12\|v\|^2-\frac1N\sum_{i=1}^N \log\cosh(v\cdot\xi_i),
		\end{equation}
		and $\tilde Z_N$ is the normalizing constant. Note that $\Phi_N$ coincides with the random potential
		considered in Gentz--L\"owe (up to notational conventions), cf.\ their definition
		of $\Phi_{N,1}$ / $\Phi_N$ (e.g.\ Lemma 3.5 in \cite{gentzloewe}).
		
		Fix $\delta\in(0,1)$.
		By the inner/intermediate/outer region analysis in the proof of Theorem~2.1 in \cite{gentzloewe}
		(one combines the bounds displayed for the intermediate region, (3.19)),
		with the outer region estimate, e.g.\ (3.17)),
		there exist deterministic constants $c_1,c_2>0$, a radius $R<\infty$, and a disorder event
		$\mathcal H_N$ with $\P_\xi(\mathcal H_N)\to1$ such that on $\mathcal H_N$,
		\begin{equation}\label{eq:quartic_lower_bound_phi}
			N\,\Phi_N(x/N^{1/4})
			\;\ge\;
			c_1\|x\|^4 - c_2\|x\|^2
			\qquad\text{for all }x\in\R^M.
		\end{equation}
		
		Let $a\in(0,c_1)$. Using \eqref{eq:HS_density_x}--\eqref{eq:quartic_lower_bound_phi} we get, on $\mathcal H_N$,
		\begin{align*}
			\E_{\widetilde Q_{N,1}}\!\left[\exp\!\big(a\|N^{1/4}u\|^4\big)\right]
			&=
			\frac{\int_{\R^M}\exp\!\big(a\|x\|^4\big)\exp\!\big(-N\Phi_N(x/N^{1/4})\big)\,dx}
			{\int_{\R^M}\exp\!\big(-N\Phi_N(x/N^{1/4})\big)\,dx} \\
			&\le
			\frac{\int_{\R^M}\exp\!\big(-(c_1-a)\|x\|^4+c_2\|x\|^2\big)\,dx}
			{\int_{B(0,R)}\exp\!\big(-N\Phi_N(x/N^{1/4})\big)\,dx}.
		\end{align*}
		The numerator is finite since $c_1-a>0$. For the denominator, $\Phi_N(0)=0$ and $\Phi_N$ is continuous,
		hence $\int_{B(0,R)}\exp(-N\Phi_N(x/N^{1/4}))\,dx \ge \mathrm{Leb}(B(0,R/2))\,e^{-C}$ for all large $N$ on $\mathcal H_N$
		(with some deterministic $C<\infty$), so the denominator is bounded below by a positive deterministic constant.
		Therefore,
		\[
		\sup_N \E_{\widetilde Q_{N,1}}\!\left[\exp\!\big(a\|N^{1/4}u\|^4\big)\mathbf 1_{\mathcal H_N}\right]
		\;<\;\infty,
		\]
		which is the claim.
	\end{proof}

	\begin{proof}[Proof of Theorem~\ref{thm:crit_poc}]
		Let
		\[
		L_N(\sigma_{1:k})
		:=\frac{d\mu_N^{(k)}}{d\pi^{\otimes k}}(\sigma_{1:k}).
		\]
		Then, by the Cauchy--Schwarz inequality
		\[
		d_{\mathrm{TV}}(\mu_N^{(k)},\pi^{\otimes k})
		=\frac12\,\E_{\pi^{\otimes k}}\big[|L_N-1|\big]
		\le \frac12\,\Big(\E_{\pi^{\otimes k}}\big[(L_N-1)^2\big]\Big)^{1/2}
		=\frac12\,\Big(\E_{\pi^{\otimes k}}[L_N^2]-1\Big)^{1/2}.
		\]
		Thus it suffices to show that
		\begin{equation}\label{eq:chi2_to_0_beta1}
			\E_{\pi^{\otimes k}}[L_N^2]\longrightarrow 1
			\qquad\text{in }\mathbb P_\xi\text{-probability}.
		\end{equation}
		
		We now use the Hubbard Stratonovich representation from Lemma~\ref{lem:HS_likelihood_beta1}
		and the exact second-moment identity from Corollary~\ref{cor:second_moment_exact_beta1}.
		Let $u,u'$ be i.i.d.\ with law $\widetilde Q_{N,1}$. Then
		\[
		\E_{\pi^{\otimes k}}[L_N^2]
		=
		\E_{u,u'\sim \widetilde Q_{N,1}}
		\exp\Bigg(\sum_{i=1}^k \Psi(u\cdot\xi_i,u'\cdot\xi_i)\Bigg)
		\]
		with
		\[
		\Psi(x,y)=\log\frac{\cosh(x+y)}{\cosh x\,\cosh y}.
		\]
		A Taylor expansion yields of $\Psi(x,y)$ aroung $(0,0)$ yields
		\begin{equation} \label{eq:Psi_Taylor}
			\Psi(x,y)=xy+R(x,y),
			\qquad
			|R(x,y)|\le C\big(x^2y^2+x^4+y^4\big),
		\end{equation}
		with a universal $C<\infty$.
		Indeed, using
		$\cosh(x+y)=\cosh(x) \cosh(y)+\sinh(x)\sinh(y)$ we obtain 
		$
		\Psi(x,y)=\log(1+\tanh(x)\tanh(y)).$ Using $\tanh(z)=z+\frac{z^3}3+O(z^5)$, we thus obtain
		$$
		\frac{\cosh(x+y)}{\cosh x\,\cosh y}=1+xy+\frac 13(x^3y+y^3x)+O(x^5,y^5).
		$$
		Using $|x|^3|y| = |x|^2 \cdot |x||y| \le |x|^2 \cdot \left(\frac{|x|^2}{2} + \frac{|y|^2}{2}\right) = \frac{|x|^4}{2} + \frac{|x|^2|y|^2}{2}$ (by arithmetic-geometric-mean) and the Taylor expansion of $\log$ gives \eqref{eq:Psi_Taylor}.

		Define
		\[
		A_N(u,u'):=\sum_{i=1}^k (u\cdot\xi_i)(u'\cdot\xi_i),
		\qquad \text{and}\quad
		B_N(u,u'):=\sum_{i=1}^k R(u\cdot\xi_i,u'\cdot\xi_i).
		\]
		Then
		\[
		\E_{\pi^{\otimes k}}[L_N^2]
		=
		\E_{u,u'\sim \widetilde Q_{N,1}}
		\exp\big(A_N(u,u')+B_N(u,u')\big).
		\]
		
		Define 
		the disorder event 
		$$\mathcal G_N:=\{\|\widehat\Sigma_k-I_M\|_{\op}\le\varepsilon_N\}$$
		with $\varepsilon_N \to 0$ sufficiently slowly
		(then $\mathbb P_\xi(\mathcal G_N)\to1$ since $M$ is fixed and $k\to\infty$).
		We work on the disorder event $\mathcal H_N$ from Lemma~\ref{lem:HS_quartic}
		intersected with $\mathcal G_N$,
		so that $$\P_\xi(\mathcal H_N\cap \mathcal G_N)\to1.$$

		By the Taylor bound \eqref{eq:Psi_Taylor}, on the disorder event $\mathcal H_N$,
		\[
		|B_N(u,u')|
		\le
		C\sum_{i=1}^k
		\Big(
		(u\cdot\xi_i)^2(u'\cdot\xi_i)^2
		+(u\cdot\xi_i)^4
		+(u'\cdot\xi_i)^4
		\Big).
		\]
		Since $\|\xi_i\|^2=M$ and $M$ is fixed, we have the deterministic bounds
		\[
		(u\cdot\xi_i)^4\le M^2\|u\|^4,
		\qquad
		(u\cdot\xi_i)^2(u'\cdot\xi_i)^2\le M^2\|u\|^2\|u'\|^2.
		\]
		Hence
		\[
		|B_N(u,u')|
		\le
		CM^2 k\big(\|u\|^4+\|u'\|^4+\|u\|^2\|u'\|^2\big).
		\]
		
		Writing $u=N^{-1/4}V$ and $u'=N^{-1/4}V'$, we obtain
		\[
		|B_N(u,u')|
		\le
		CM^2\,\frac{k}{N}
		\big(\|V\|^4+\|V'\|^4+\|V\|^2\|V'\|^2\big).
		\]
		Hence using $k=o(\sqrt N)$ so that $k/N\to0$,
		together with the quartic integrability from Lemma~\ref{lem:HS_quartic}
		\begin{equation}\label{eq:BN_small_beta1}
			\E_{\widetilde Q_{N,1}^{\otimes2}}\big[|B_N(u,u')|\mathbf 1_{\mathcal H_N}\big]
			\;\le\;
			C\,M^2\,\frac{k}{N}\;
			\E\big[\|V\|^4+\|V'\|^4+\|V\|^2\|V'\|^2\big]
			\;=\;o(1).
		\end{equation}
		
		In particular, $B_N(u,u')\to 0$ in $\widetilde Q_{N,1}^{\otimes2}$-probability on $\mathcal H_N$.
		
		For the main term, write
		\[
		A_N(u,u')
		=
		u^\top\Big(\sum_{i=1}^k \xi_i\xi_i^\top\Big)u'.
		\]
		On $\mathcal G_N$ we have $\frac1k\sum_{i=1}^k\xi_i\xi_i^\top = I_M+o(1)$ in operator norm, hence
		\begin{equation}\label{eq:ANbeta1}
			A_N(u,u')
			=
			k\,u\cdot u' + o(1)\,k\|u\|\,\|u'\|.
		\end{equation}
		Recalling $u=N^{-1/4}V$ we obtain
		\[
		k\,u\cdot u' = \frac{k}{\sqrt N}\,V\cdot V'
		\qquad \text{and thus} \quad
		k\|u\|\,\|u'\| = \frac{k}{\sqrt N}\,\|V\|\,\|V'\|.
		\]
		Since $k/\sqrt N\to 0$, Lemma~\ref{lem:HS_quartic} implies that both quantities converge to $0$
		in $\widetilde Q_{N,1}^{\otimes2}$-probability on $\mathcal H_N$.
		Indeed, $$c\|V\|^2\le \delta\|V\|^4 + c^2/(4\delta),$$ hence
		$\exp(c\|V\|^2)\le e^{c^2/(4\delta)}\exp(\delta\|V\|^4)$ and apply Lemma~\ref{lem:HS_quartic}.

		Therefore,
		\begin{equation}\label{eq:AN_small_beta1}
			A_N(u,u')\xrightarrow[N\to\infty]{} 0
			\quad\text{in }\widetilde Q_{N,1}^{\otimes2}\text{-probability on }\mathcal H_N.
		\end{equation}
		
		Finally, we show that the exponential may be passed through the limit.
		By Cauchy--Schwarz and $|V\cdot V'|\le \frac12(\|V\|^2+\|V'\|^2)$,
		\[
		\exp\big(|k\,u\cdot u'|\big)
		=
		\exp\!\Big(\frac{k}{\sqrt N}|V\cdot V'|\Big)
		\le
		\exp\!\Big(\frac{k}{2\sqrt N}\|V\|^2\Big)\,
		\exp\!\Big(\frac{k}{2\sqrt N}\|V'\|^2\Big).
		\]
		As $k/\sqrt N\to0$, for all $N$ large enough we have $k/\sqrt N\le 1$, hence
		$$\exp(\frac{k}{2\sqrt N}\|V\|^2)\le \exp(\frac12\|V\|^2).$$
		
		Lemma~\ref{lem:HS_quartic} implies $\sup_N \E[\exp(c\|V\|^2)\mathbf 1_{\mathcal H_N}]<\infty$
		for every fixed $c<\infty$, so the family
		$\exp(A_N(u,u')+B_N(u,u'))$ is uniformly integrable on $\mathcal H_N$.
		Combining \eqref{eq:AN_small_beta1} and \eqref{eq:BN_small_beta1} (note that still $k/N \to 0$), we conclude that
		\[
		\E_{u,u'\sim \widetilde Q_{N,1}}\big[\exp(A_N(u,u')+B_N(u,u'))\mathbf 1_{\mathcal H_N}\big]
		\longrightarrow 1.
		\]
		Since $\P_\xi(\mathcal H_N)\to1$, this proves \eqref{eq:chi2_to_0_beta1}, and hence
		$d_{\mathrm{TV}}(\mu_N^{(k)},\pi^{\otimes k})\to0$ in $\P_\xi$-probability.
	\end{proof}
	
	\section{Proof of Theorem \ref{thm:crit_break_sqrtN}}
	
	Next we show that propagation of chaos starts to fail for finite $M$ and $k=cN^{1/2}$ for $c>0$.

	\begin{proof}[Proof of Theorem~\ref{thm:crit_break_sqrtN}]
		We prove Theorem~\ref{thm:crit_break_sqrtN} in several steps.
		
		We start with a familiar first step.
		As for above define the density of the marginal distribution with respect to the Rademacher product measure:  
		\[
		L_N(\sigma_{1:k}):=\frac{d\mu_N^{(k)}}{d\pi^{\otimes k}}(\sigma_{1:k}).
		\]
		Recall that 
		\[
		d_{\mathrm{TV}}(\mu_N^{(k)},\pi^{\otimes k})
		=\frac12\,\E_{\pi^{\otimes k}}\big[|L_N-1|\big].
		\]
		Hence it suffices to show that $L_N$ does not converge to $1$ in
		$\pi^{\otimes k}$-probability, uniformly on a high-probability disorder event (note that, of course, $\E_{\pi^{\otimes k}}(L_N)=1$ for all disorder events).
		A convenient sufficient way is to show give a lower bound und $\E[L^2]$ and an upper bound 
		on some $p'th$ moment of $|L_N-1|$ which we will see in the next Lemma.

		\begin{lemma}
			\label{lem:chi2_to_tv_Lp}
			Let $(\Omega,\mathcal F,\pi)$ be a probability space and let $L\ge 0$ with $\E_\pi[L]=1$.
			Fix $\eta>0$ and set $p:=2+\eta$. Assume that
			\[
			\E_\pi[L^2]\ge 1+\kappa
			\qquad\text{and}\qquad
			\E_\pi\big[|L-1|^{p}\big]\le C_p
			\]
			for some $\kappa>0$ and $C_p<\infty$. Then
			\[
			d_{\mathrm{TV}}(L\pi,\pi)=\frac12\,\E_\pi|L-1|
			\ \ge\ 
			\frac12 \kappa^{(p-1)/(p-2)} C_p^{-1/(p-2)}
			=:
			b(\kappa,C_p,\eta)\ >0.
			\]
		\end{lemma}
		
		\begin{proof}
			Let $X:=|L-1|$. 
			Define, $\alpha=\frac{p-2}{2(p-1)}$.
			Then, by a generalized H\"older inequality, $$\|X\|_2 \le \|X\|_1^{\alpha}\|X\|_p^{1-\alpha},$$ hence
			\[
			\|X\|_1 \ge \|X\|_2^{1/\alpha}\,\|X\|_p^{-(1-\alpha)/\alpha}.
			\]
			Note that $\|X\|_2^2=\E[(L-1)^2]=\E[L^2]-1\ge\kappa$.
			Moreover, $\|X\|_p \le C_p^{1/p}$ since $\|X\|_p^p=\E[X^p]\le C_p$.
			With $\alpha = \frac{p-2}{2(p-1)}$ we have $\frac{1}{2\alpha}=\frac{p-1}{p-2}$ and
			$\frac{1-\alpha}{\alpha p}=\frac{1}{p-2}$.
			So, we obtain
			\[
			\|X\|_1 \ge \kappa^{(p-1)/(p-2)}\, C_p^{-1/(p-2)}.
			\]
		\end{proof}

		By Lemma~\ref{lem:chi2_to_tv_Lp}, it therefore suffices to show that there exist
		constants $\kappa>0$, $\eta>0$ and $C_p<\infty$ (with $p=2+\eta$) such that,
		on a disorder event $\mathcal H_N$ with $\mathbb P_\xi(\mathcal H_N)\to 1$,
		\begin{align}
			\text{on }\mathcal H_N:\quad \E_{\pi^{\otimes k}}[L_N^2]\ge 1+\kappa,
			\label{eq:chi2_gap_goal}\\
			\text{and on } \mathcal H_N:\quad \E_{\pi^{\otimes k}}[|L_N-1|^p]\le C_p.
			\label{eq:Lp_goal}
		\end{align}
		We will prove \eqref{eq:chi2_gap_goal} and \eqref{eq:Lp_goal}below.

		\medskip
		
		We start with \eqref{eq:chi2_gap_goal}.
		
		Recall from Corollary~\ref{cor:second_moment_exact_beta1} that
		\[
		\E_{\pi^{\otimes k}}[L_N^2]
		=
		\E_{u,u'\sim \widetilde Q_{N,1}}
		\exp\Bigg(
		\sum_{i=1}^k
		\Psi\big(u\cdot\xi_i,\;u'\cdot\xi_i\big)
		\Bigg),
		\]
		where
		\[
		\Psi(x,y)
		:=
		\log\frac{\cosh(x+y)}{\cosh x\,\cosh y}.
		\]
		
		By \eqref{eq:Psi_Taylor}
		we can again rewrite
		\[
		\sum_{i=1}^k \Psi(u\cdot\xi_i,u'\cdot\xi_i)
		=
		\underbrace{\sum_{i=1}^k (u\cdot\xi_i)(u'\cdot\xi_i)}_{=:A_N(u,u')}
		+
		\underbrace{\sum_{i=1}^k R(u\cdot\xi_i,u'\cdot\xi_i)}_{=:B_N(u,u')}.
		\]
		
		As in \eqref{eq:BN_small_beta1} (again $k/N \to 0$)
		\[
		\E_{\widetilde Q_{N,1}^{\otimes2}}
		\big[|B_N(u,u')|\,\mathbf 1_{\mathcal H_N}\big]\xrightarrow[N\to\infty]{}0.
		\]
		and hence by Markov' inequality
		$
		B_N(u,u')\xrightarrow[N\to\infty]{}0
		\;\text{in } \widetilde Q_{N,1}^{\otimes2}\text{-probability on }\mathcal H_N.
		$
		
		\medskip
		For  $A_N(u,u')$ again write
		$
		A_N(u,u')
		=
		u^\top\Big(\sum_{i=1}^k \xi_i\xi_i^\top\Big)u'
		$
		and use that
		on a disorder event of probability tending to one,
		$
		\frac{1}{k}\sum_{i=1}^k \xi_i\xi_i^\top
		=
		I_M+o(1).
		$

		On the event $\{\|\widehat\Sigma_k-I_M\|_{\op}=o(1)\}$ we have as in \eqref{eq:ANbeta1}
		\[
		A_N(u,u')=k\,u\cdot u' + o(1)\,k\|u\|\,\|u'\|.
		\]
		Since $k\|u\|\,\|u'\| = (k/\sqrt N)\,\|V\|\,\|V'\|$ and $k/\sqrt N=O(1)$,
		Lemma~\ref{lem:HS_quartic} implies $k\|u\|\,\|u'\|=O_{\widetilde Q^{\otimes2}}(1)$,
		hence the error is $o_{\widetilde Q^{\otimes2}}(1)$.
		\medskip
		
		%
		From the previous step we have, on a disorder event of probability
		tending to one,
		\[
		\E_{\pi^{\otimes k}}[L_N^2]
		=
		\E_{u,u'\sim \widetilde Q_{N,1}}
		\exp\!\big(k\,u\cdot u' + r_N(u,u')\big)
		+ o(1),
		\]
		where $r_N(u,u')\to0$ in
		$\widetilde Q_{N,1}^{\otimes2}$-probability and uniformly on sets
		$\{\|u\|,\|u'\|\le C N^{-1/4}\}$.
		Writing $u=N^{-1/4}V$ and $u'=N^{-1/4}V'$, we have
		$k\,u\cdot u' = kN^{-1/2}V\cdot V' \to c\,V\cdot V'$.
		To justify uniform integrability, note that
		\[
		|V\cdot V'|
		\le \frac12\big(\|V\|^2+\|V'\|^2\big).
		\]
		For every $\delta>0$ we further have
		\[
		\|V\|^2
		\le \delta \|V\|^4 + \frac{1}{4\delta},
		\]
		hence
		\[
		|V\cdot V'|
		\le
		\delta\big(\|V\|^4+\|V'\|^4\big)
		+\frac{1}{2\delta}.
		\]
		Therefore, for $c\ge0$,
		\[
		\exp\big(c\,|V\cdot V'|\big)
		\le
		\exp\!\Big(\frac{c}{2\delta}\Big)
		\exp\!\big(c\delta \|V\|^4\big)
		\exp\!\big(c\delta \|V'\|^4\big).
		\]
		Choosing $\delta>0$ small enough and using Lemma~\ref{lem:HS_quartic},
		the right-hand side has uniformly bounded expectation.
		Thus the family $\exp(c\,V\cdot V')$ is uniformly integrable.
		
		Hence
		\[
		\E_{\pi^{\otimes k}}[L_N^2]
		=
		\E_{u,u'\sim \widetilde Q_{N,1}}
		\exp\!\big(k\,u\cdot u'\big)
		+ o(1),
		\qquad
		\text{in }\mathbb P_\xi\text{-probability}.
		\]

		Since $k/\sqrt N\to c>0$ and the limiting law of $V$ is non-degenerate
		(see \cite{gentzloewe}),
		the random variable $V\cdot V'$ is not almost surely equal to $0$.
		Hence, for every $c>0$,
		\[
		\E\big[\exp(c\,V\cdot V')\big] > 1.
		\]
		Consequently,
		\[
		\liminf_{N\to\infty}\E_{\pi^{\otimes k}}[L_N^2] > 1.
		\]
		
		This proves \eqref{eq:chi2_gap_goal}.
		
		\medskip
		
		We will next show \eqref{eq:Lp_goal} i.e.
		$$	\E_{\pi^{\otimes k}}\big[|L_N-1|^{p}\big]\mathbf 1_{\mathcal H_N} \le C_p$$ with $p:=3$.
		Since by convexity of the p'th power $|x-1|^p \le 2^{p-1}(x^p+1)$ for $x\ge0$, it suffices to show that
		there exists a deterministic constant $C_p<\infty$ such that, on a disorder
		event $\mathcal H_N$ with $\mathbb P_\xi(\mathcal H_N)\to1$,
		\begin{equation}\label{eq:LN_p_bound_goal}
			\sup_N \E_{\pi^{\otimes k}}\!\big[L_N^p\,\mathbf 1_{\mathcal H_N}\big]\le C_p.
		\end{equation}
		
		\medskip
		
		\noindent
		Recall that 
		by Lemma~\ref{lem:HS_likelihood_beta1} 
		\[
		L_N(\sigma_{1:k})
		=
		\E_{u\sim\widetilde Q_{N,1}}
		\exp\Big(\sum_{i=1}^k(\sigma_i u\cdot\xi_i-\log\cosh(u\cdot\xi_i))\Big).
		\]
		Hence, for integer $p$,
		\[
		L_N(\sigma_{1:k})^p
		=
		\E_{u^{(1)},\dots,u^{(p)}\ \mathrm{i.i.d.}\sim\widetilde Q_{N,1}}
		\exp\Bigg(
		\sum_{i=1}^k\Big(
		\sigma_i \Big(\sum_{a=1}^p u^{(a)}\Big)\!\cdot\xi_i
		-
		\sum_{a=1}^p\log\cosh(u^{(a)}\cdot\xi_i)
		\Big)\Bigg)
		\]
		(which is the usual replica representation of moments).
		Taking expectation over $\sigma\sim\pi^{\otimes k}$ and using
		$\E_\pi[e^{\sigma t}]=\cosh(t)$ yields
		\begin{equation}\label{eq:Ep_LNp_exact}
			\E_{\pi^{\otimes k}}[L_N^p]
			=
			\E_{u^{(1)},\dots,u^{(p)}}
			\exp\Bigg(
			\sum_{i=1}^k
			\Big[
			\log\cosh\Big(\Big(\sum_{a=1}^p u^{(a)}\Big)\!\cdot\xi_i\Big)
			-
			\sum_{a=1}^p\log\cosh(u^{(a)}\cdot\xi_i)
			\Big]
			\Bigg),
		\end{equation}
		where the outer expectation is with respect to $\widetilde Q_{N,1}^{\otimes p}$.
		
		\medskip
		
		\noindent
		To simplify notation
		define, for $x_1,\dots,x_p\in\mathbb R$,
		\[
		\Theta_p(x_1,\dots,x_p)
		:=
		\log\cosh\Big(\sum_{a=1}^p x_a\Big)-\sum_{a=1}^p\log\cosh(x_a).
		\]
		Since $\log\cosh$ is convex and even, one has the elementary bound
		\begin{equation}\label{eq:Theta_quad_bound}
			\Theta_p(x_1,\dots,x_p)
			\le
			\sum_{1\le a<b\le p} x_a x_b
			\le
			\frac{p-1}{2}\sum_{a=1}^p x_a^2,
		\end{equation}
		valid for all $(x_1,\dots,x_p)\in\mathbb R^p$.
		The first inequality follows from $(\log \cosh(x))''\le 1$ by induction over $p$ using
		$$
		f(x_1+x_2)-f(x_1)-f(x_2)+f(0)=\int_0^{x_1}\int_0^{x_2} f''(s+t)dt\le x_1 x_2
		$$
		for the base case.
		The second inequality is just $2x_ax_b\le x_a^2+x_b^2$.
		
		Applying \eqref{eq:Theta_quad_bound} with $x_a=u^{(a)}\cdot\xi_i$ gives
		\[
		\log\cosh\Big(\Big(\sum_{a=1}^p u^{(a)}\Big)\!\cdot\xi_i\Big)
		-
		\sum_{a=1}^p\log\cosh(u^{(a)}\cdot\xi_i)
		\le
		\frac{p-1}{2}\sum_{a=1}^p (u^{(a)}\cdot\xi_i)^2.
		\]
		Summing over $i\le k$ yields
		\begin{equation}\label{eq:sum_site_quad_bound}
			\sum_{i=1}^k
			\Big[
			\log\cosh\Big(\Big(\sum_{a=1}^p u^{(a)}\Big)\!\cdot\xi_i\Big)
			-
			\sum_{a=1}^p\log\cosh(u^{(a)}\cdot\xi_i)
			\Big]
			\le
			\frac{p-1}{2}\sum_{a=1}^p \sum_{i=1}^k (u^{(a)}\cdot\xi_i)^2.
		\end{equation}
		
		\medskip
		
		Again define 
		the disorder event $$\mathcal G_N:=\{\|\widehat\Sigma_k-I_M\|_{\op}\le\varepsilon\}$$
		(with $\mathbb P_\xi(\mathcal G_N)\to1$ since $M$ is fixed and $k\to\infty$).
		On $\mathcal G_N$ we have
		\[
		\sum_{i=1}^k (u^{(a)}\cdot\xi_i)^2
		=
		k\,(u^{(a)})^\top \widehat\Sigma_k\,u^{(a)}
		\le
		k(1+\varepsilon)\|u^{(a)}\|^2.
		\]
		Inserting this into \eqref{eq:sum_site_quad_bound} and then into \eqref{eq:Ep_LNp_exact},
		we obtain on $\mathcal G_N$:
		\[
		\E_{\pi^{\otimes k}}[L_N^p]
		\le
		\E_{\widetilde Q_{N,1}^{\otimes p}}
		\exp\Big(
		C\,k\sum_{a=1}^p \|u^{(a)}\|^2
		\Big),
		\]
		where $C=C(p,\varepsilon)$ is deterministic.
		
		Writing $u^{(a)}=N^{-1/4}V^{(a)}$, we have
		\[
		k\|u^{(a)}\|^2 = \frac{k}{\sqrt N}\,\|V^{(a)}\|^2.
		\]
		Since $k/\sqrt N\to c$, there exists a deterministic $C'<\infty$ such that
		$k/\sqrt N\le C'$ for all $N$ large, and hence
		\[
		\E_{\pi^{\otimes k}}[L_N^p]
		\le
		\E_{\widetilde Q_{N,1}^{\otimes p}}
		\exp\Big(
		C'\sum_{a=1}^p \|V^{(a)}\|^2
		\Big)
		=
		\prod_{a=1}^p
		\E_{\widetilde Q_{N,1}}
		\exp\big(C'\|V\|^2\big).
		\]
		
		By Lemma~\ref{lem:HS_quartic}, there exist $a>0$ and $C<\infty$ such that
		\[
		\sup_N \E_{\widetilde Q_{N,1}}\big[\exp(a\|V\|^4)\mathbf 1_{\mathcal H_N}\big]\le C.
		\]
		Applying the elementary inequality
		\[
		C' r^2 \le \delta r^4 + \frac{(C')^2}{4\delta},
		\qquad r\ge0,
		\]
		with $r:=\|V\|$
		and choosing $\delta\le a$, we obtain
		\[
		\exp(C'\|V\|^2)
		\le
		e^{(C')^2/(4\delta)}\,\exp(\delta\|V\|^4).
		\]
		Taking expectations yields
		\[
		\sup_N \E_{\widetilde Q_{N,1}}\big[\exp(C'\|V\|^2)\mathbf 1_{\mathcal H_N}\big]<\infty,
		\]
		which proves \eqref{eq:LN_p_bound_goal}.
		
		\medskip
		
		Now we are ready to combine the previous arguments: 
		
		Let $p=3$ be as above, and let $\mathcal H_N$ be the disorder event
		(on which Lemma~\ref{lem:HS_quartic} holds) intersected with the covariance
		concentration event $\mathcal G_N$. Then
		$\mathbb P_\xi(\mathcal H_N)\to 1$.
		
		We established that there exists $\kappa=\kappa(c,M)>0$ such that on $\mathcal H_N$,
		\[
		\E_{\pi^{\otimes k}}[L_N^2] \ge 1+\kappa
		\qquad\text{for all $N$ large enough.}
		\]
		We also saw that there exists $C_p<\infty$ such that on $\mathcal H_N$,
		\[
		\E_{\pi^{\otimes k}}\big[|L_N-1|^{p}\big] \le C_p
		\qquad\text{for all $N$ large enough.}
		\]
		Applying Lemma~\ref{lem:chi2_to_tv_Lp} with $L=L_N$ on $\mathcal H_N$ yields
		\[
		d_{\mathrm{TV}}(\mu_N^{(k)},\pi^{\otimes k})
		=
		\frac12\,\E_{\pi^{\otimes k}}|L_N-1|
		\ge
		b(\kappa,C_p,\eta)
		=: b(c,M) >0
		\qquad\text{on }\mathcal H_N,
		\]
		for all $N$ large. Since $\mathbb P_\xi(\mathcal H_N)\to1$, this proves
		\[
		\liminf_{N\to\infty} d_{\mathrm{TV}}(\mu_N^{(k)},\pi^{\otimes k}) \ge b(c,M)
		\qquad\text{in }\mathbb P_\xi\text{-probability},
		\]
		and completes the proof of Theorem~\ref{thm:crit_break_sqrtN}.
	\end{proof}
%
	

	\appendix
	
	\section{Exponential moments for Rademacher quadratic forms}
	\label{app:mgf_quad}
	
	In this appendix we prove an exponential-moment bound for centered quadratic forms
	in independent Rademacher variables. The argument is a standard Laplace-transform
	derivation from the Hanson--Wright inequality, included for completeness.
	
	\begin{lemma}[Hanson--Wright inequality for Rademacher variables]
		\label{lem:HW_rademacher}
		Let $\sigma=(\sigma_1,\dots,\sigma_k)$ have independent Rademacher coordinates and
		let $A$ be a real symmetric $k\times k$ matrix. Set
		\[
		X:=\sigma^\top A\sigma-\tr(A).
		\]
		There exist absolute constants $c_0,C_0\in(0,\infty)$ such that for all $t\ge 0$,
		\begin{equation}\label{eq:HW_tail}
			\mathbb P\big(|X|\ge t\big)
			\le
			2\exp\!\left(-c_0\min\Big\{\frac{t^2}{\|A\|_F^2},\frac{t}{\|A\|_{\op}}\Big\}\right).
		\end{equation}
	\end{lemma}
	
	\begin{proof}
		This is the Hanson--Wright inequality for subgaussian vectors applied to the
		Rademacher vector $\sigma$ (note thet for Rademachers $\tr(A)=\E \sigma^\top A \sigma$); see, e.g., Vershynin,
		\emph{High-Dimensional Probability}, Thm.~6.2.1.
	\end{proof}
	
	\begin{lemma}
		\label{lem:mgf_quad}
		Let $\sigma$ and $A$ be as in Lemma~\ref{lem:HW_rademacher}, and let
		$X=\sigma^\top A\sigma-\tr(A)$.
		There exist absolute constants $c,C\in(0,\infty)$ such that for all
		$|s|\le c/\|A\|_{\op}$,
		\begin{equation}\label{eq:mgf_quad_bound}
			\E\exp(sX)\le \exp\big(C s^2\|A\|_F^2\big).
		\end{equation}
	\end{lemma}
	
	\begin{proof}
		By symmetry of $\sigma$ we have $X\stackrel{d}{=}-X$, hence
		$\E e^{sX}=\E e^{|s|X}$ for $s\in\mathbb R$. It therefore suffices to treat $s\ge 0$.
		
		For any real random variable $Y$ and $s\ge 0$, integration by parts yields
		\begin{equation}\label{eq:mgf_tail_identity}
			\E e^{sY}
			=
			1+s\int_0^\infty e^{st}\,\mathbb P(Y\ge t)\,dt
			+s\int_0^\infty e^{-st}\,\mathbb P(Y\le -t)\,dt.
		\end{equation}
		Applying this with $Y=X$ and using $\mathbb P(X\le -t)=\mathbb P(-X\ge t)=\mathbb P(X\ge t)$,
		we obtain
		\begin{equation}\label{eq:mgf_tail_sym}
			\E e^{sX}
			=
			1+2s\int_0^\infty e^{st}\,\mathbb P(X\ge t)\,dt.
		\end{equation}
		Moreover, $\mathbb P(X\ge t)\le \mathbb P(|X|\ge t)$, hence by
		Lemma~\ref{lem:HW_rademacher},
		\begin{equation}\label{eq:tail_use}
			\mathbb P(X\ge t)
			\le
			2\exp\!\left(-c_0\min\Big\{\frac{t^2}{\|A\|_F^2},\frac{t}{\|A\|_{\op}}\Big\}\right).
		\end{equation}
		Plugging \eqref{eq:tail_use} into \eqref{eq:mgf_tail_sym} gives
		\begin{equation}\label{eq:mgf_int_start}
			\E e^{sX}
			\le
			1+4s\int_0^\infty
			\exp\!\left(st-c_0\min\Big\{\frac{t^2}{\|A\|_F^2},\frac{t}{\|A\|_{\op}}\Big\}\right)\,dt.
		\end{equation}
		
		Next, let
		\[
		t_0:=\frac{\|A\|_F^2}{\|A\|_{\op}},
		\]
		so that $\frac{t^2}{\|A\|_F^2}\le \frac{t}{\|A\|_{\op}}$ for $0\le t\le t_0$, and the
		reverse inequality holds for $t\ge t_0$.
		We split the integral in \eqref{eq:mgf_int_start} accordingly:
		\begin{multline*}
			\int_0^\infty \exp\!\left(st-c_0\min\Big\{\frac{t^2}{\|A\|_F^2},\frac{t}{\|A\|_{\op}}\Big\}\right)\,dt\,dt
			\\=
			\int_0^{t_0} \exp\!\left(st-c_0\min\Big\{\frac{t^2}{\|A\|_F^2},\frac{t}{\|A\|_{\op}}\Big\}\right)\,dt)\,dt+\int_{t_0}^\infty \exp\!\left(st-c_0\min\Big\{\frac{t^2}{\|A\|_F^2},\frac{t}{\|A\|_{\op}}\Big\}\right)\,dt\,dt
			=:I_1+I_2.
		\end{multline*}

		For $t\in[0,t_0]$ we have
		$\min\{\frac{t^2}{\|A\|_F^2},\frac{t}{\|A\|_{\op}}\}=\frac{t^2}{\|A\|_F^2}$, hence
		\[
		I_1
		=
		\int_0^{t_0}\exp\!\left(st-c_0\frac{t^2}{\|A\|_F^2}\right)\,dt
		\le
		\int_0^\infty\exp\!\left(st-c_0\frac{t^2}{\|A\|_F^2}\right)\,dt.
		\]
		Completing the square yields
		$
		st-c_0\frac{t^2}{\|A\|_F^2}
		=
		-\frac{c_0}{\|A\|_F^2}\Big(t-\frac{s\|A\|_F^2}{2c_0}\Big)^2
		+\frac{s^2\|A\|_F^2}{4c_0},
		$
		and therefore
		\begin{align}
			I_1
			&\le
			\exp\!\Big(\frac{s^2\|A\|_F^2}{4c_0}\Big)
			\int_0^\infty
			\exp\!\left(-\frac{c_0}{\|A\|_F^2}\Big(t-\frac{s\|A\|_F^2}{2c_0}\Big)^2\right)\,dt
			\notag\\
			&\le
			\exp\!\Big(\frac{s^2\|A\|_F^2}{4c_0}\Big)
			\int_{-\infty}^{\infty}
			\exp\!\left(-\frac{c_0}{\|A\|_F^2}u^2\right)\,du
			=	\sqrt{\frac{\pi}{c_0}}\;\|A\|_F\;
			\exp\!\Big(\frac{s^2\|A\|_F^2}{4c_0}\Big).
			\label{eq:I1_bound}
		\end{align}
		%
		For $t\ge t_0$ we have
		$\min\{\frac{t^2}{\|A\|_F^2},\frac{t}{\|A\|_{\op}}\}=\frac{t}{\|A\|_{\op}}$, hence
		\[
		I_2
		=
		\int_{t_0}^\infty \exp\!\left(st-c_0\frac{t}{\|A\|_{\op}}\right)\,dt
		=
		\int_{t_0}^\infty \exp\!\left(-(c_0/\|A\|_{\op}-s)t\right)\,dt.
		\]
		Assume $0\le s\le \frac{c_0}{2\|A\|_{\op}}$, so that
		$c_0/\|A\|_{\op}-s\ge c_0/(2\|A\|_{\op})$. Then
		\begin{equation}\label{eq:I2_bound}
			I_2
			\le
			\int_{t_0}^\infty \exp\!\left(-\frac{c_0}{2\|A\|_{\op}}t\right)\,dt
			=
			\frac{2\|A\|_{\op}}{c_0}\,
			\exp\!\left(-\frac{c_0}{2\|A\|_{\op}}t_0\right)
			=
			\frac{2\|A\|_{\op}}{c_0}\,
			\exp\!\left(-\frac{c_0}{2}\frac{\|A\|_F^2}{\|A\|_{\op}^2}\right).
		\end{equation}
		In particular, $I_2\le 2\|A\|_{\op}/c_0$ for all such $s$.
		
		Combining \eqref{eq:mgf_int_start}, \eqref{eq:I1_bound}, and \eqref{eq:I2_bound},
		for $0\le s\le c_0/(2\|A\|_{\op})$ we get
		\[
		\E e^{sX}
		\le
		1+4s\Big(
		\sqrt{\frac{\pi}{c_0}}\;\|A\|_F\;
		e^{\frac{s^2\|A\|_F^2}{4c_0}}
		+\frac{2\|A\|_{\op}}{c_0}
		\Big).
		\]
		Since $s\|A\|_{\op}\le c_0/2$ in this range, the second term is bounded by an
		absolute constant. Moreover, for $x\ge 0$ one has $1+x\le e^{x}$, so
		\[
		\E e^{sX}
		\le
		\exp\!\Big(C_1 s\|A\|_F\,e^{\frac{s^2\|A\|_F^2}{4c_0}} + C_2\Big).
		\]
		Finally, using $s\|A\|_F\le 1+s^2\|A\|_F^2$ and adjusting constants, we obtain
		\[
		\E e^{sX}
		\le
		\exp\!\big(C s^2\|A\|_F^2\big)
		\]
		for absolute constants $c:=c_0/2$ and $C<\infty$, which proves
		\eqref{eq:mgf_quad_bound} for $s\ge 0$. By the reduction to $|s|$ at the start,
		it holds for all $|s|\le c/\|A\|_{\op}$.
	\end{proof}

\end{document}